\documentclass{amsart}
\usepackage{lscape}
\usepackage{amsmath, amscd}
\usepackage{stmaryrd}
\usepackage{mathrsfs}
\usepackage{amssymb}
\usepackage{amsfonts}
\usepackage{caption} 
\usepackage{tensor}
\usepackage[all]{xy}
\usepackage{enumerate}
\usepackage{tikz}
\usepackage{adjustbox}
\usepackage{graphicx}
\usetikzlibrary{matrix, arrows}
\usetikzlibrary{positioning}
\usetikzlibrary{decorations.pathreplacing}
\numberwithin{equation}{subsection}
\theoremstyle{plain}
\newtheorem{thm}[subsection]{Theorem}
\newtheorem{prop}[subsection]{Proposition}
\newtheorem{lemma}[subsection]{Lemma}
\newtheorem{cor}[subsection]{Corollary}
\theoremstyle{definition}

\theoremstyle{remark}
\newtheorem{rem}[subsection]{Remark}

\newtheorem{final remark}[subsection]{Final Remark}
\newtheorem{example}[subsection]{Example}
\makeatletter
\newcommand*\bigcdot{\mathpalette\bigcdot@{.65}}
\newcommand*\bigcdot@[2]{\mathbin{\vcenter{\hbox{\scalebox{#2}{$\m@th#1\bullet$}}}}}
\makeatother
\begin{document}
\title[Motivic cohomology and $K$-theory of some surfaces]{Motivic cohomology and $K$-theory of some surfaces over finite fields}
\author{Oliver Gregory}
\begin{abstract}
We compute the algebraic $K$-theory of some classes of surfaces defined over finite fields. We achieve this by first calculating the motivic cohomology groups and then studying the motivic Atiyah-Hirzebruch spectral sequence. In an appendix, we slightly enlarge the class of surfaces for which Parshin's conjecture is known. 
\end{abstract}
\address{Huxley Building, Imperial College London, South Kensington, London, UK, SW7 2AZ}
\email {o.gregory22@imperial.ac.uk}
\date{August 03, 2023 \\ This research was supported by EPSRC grant EP/T005351/1 and the Heilbronn Institute for Mathematical Research.}
\maketitle
\pagestyle{myheadings}

\section{Introduction}
The algebraic $K$-groups of schemes are rich invariants with deep relationships to arithmetic, geometry and much more. The easiest setting to study $K$-theory for schemes is when the base field is finite. In dimension zero, the $K$-groups of finite fields were computed by Quillen \cite{Qui72}; one has
\[ K_{n}(\mathbb{F}_{q})\simeq \begin{cases} 
      \mathbb{Z} & \text{ if }n=0 \\
      0 & \text{ if } n=2m, m\geq 1 \\
      \mathbb{Z}/(q^{m}-1)\mathbb{Z} & \text{ if }n=2m-1, m\geq 1\,. \\
   \end{cases}
\]
In dimension one, the $K$-groups of a smooth projective curve $X$ over $\mathrm{Spec}\,\mathbb{F}_{q}$ are well-known (see \cite[Theorem 54]{Wei05}); let $\mathrm{char}(\mathbb{F}_{q})=p$, then one has
\[ K_{n}(X)\simeq \begin{cases} 
      \mathbb{Z}\oplus\mathrm{Pic}(X) & \text{ if }n=0 \\
      \displaystyle\bigoplus_{\ell\neq p}{J(X_{\overline{\mathbb{F}}_{q}})}_{\ell\mathrm{-tors}}(m)^{\Gamma} & \text{ if } n=2m, m\geq 1 \\
      \left(\mathbb{Z}/(q^{m}-1)\mathbb{Z}\right)^{\oplus 2} & \text{ if }n=2m-1, m\geq 1 \\
   \end{cases}
\]
where $J(X_{\overline{\mathbb{F}}_{q}})_{\ell\mathrm{-tors}}$ denotes the $\ell$-primary torsion subgroup of the group of points on the Jacobian of $X_{\overline{\mathbb{F}}_{q}}$, and $\Gamma=\mathrm{Gal}(\overline{\mathbb{F}}_{q}/\mathbb{F}_{q})$. In general, it is a conjecture of Parshin that the higher $K$-groups $K_{n}(X)$, $n\geq 1$, of a smooth projective variety $X$ over a finite field are torsion. 

\begin{rem}\label{Kahn on Parshin}
Parshin's conjecture remains wide open, but there are some partial results. Geisser \cite[Theorem 3.3]{Gei98} and Kahn \cite[Corollaire 2.2]{Kah03} have shown that if $X$ satisfies the Tate conjecture (i.e. the $\ell$-adic cycle class map
\begin{equation*}
\mathrm{CH}^{i}(X)\otimes\mathbb{Q}_{\ell}\rightarrow H^{2i}_{\mathrm{\acute{e}t}}(X_{\overline{\mathbb{F}}_{q}},\mathbb{Q}_{\ell}(i))^{\Gamma}
\end{equation*}
is surjective for all primes $\ell\neq p$ and for all $i\geq 0$) and Beilinson's conjecture (i.e. numerical and rational equivalence on $X$ agree rationally), then $X$ satisfies Parshin's conjecture. Let $\mathcal{M}_{\mathrm{Ab}}$ denote the set of isomorphism classes of smooth projective varieties $X$ over $\mathbb{F}_{q}$ whose rational Chow motives lie in the sub-category of Chow motives generated by abelian varieties and Artin motives. For example, $\mathcal{M}_{\mathrm{Ab}}$ contains curves, abelian varieties, unirational varieties of dimension $\leq 3$, and products thereof. If $X\in\mathcal{M}_{\mathrm{Ab}}$ then one says that $X$ is of ``abelian-type''. Kahn \cite[Corollaire 2.1]{Kah03} has shown that if $X\in\mathcal{M}_{\mathrm{Ab}}$ and $X$ satisfies the Tate conjecture, then $X$ satisfies Beilinson's conjecture. What is more, Soul\'{e} \cite[Th\'{e}or\`{e}me 4(i)]{Sou84} has shown that if $X\in\mathcal{M}_{\mathrm{Ab}}$ and $\dim X\leq 3$ then $X$ satisfies the Tate conjecture. In particular, Parshin's conjecture is known to hold for any $X\in\mathcal{M}_{\mathrm{Ab}}$ with $\dim X\leq 3$ \cite[Corollaire 2.2]{Kah03}. In Appendix \ref{Parshin section} we shall show that Parshin's conjecture also holds for surfaces admitting a rational decomposition of the diagonal. We shall show as well that K3 surfaces $X$ over $\mathrm{Spec}\,\mathbb{F}_{q}$ with geometric Picard rank $\rho(X_{\overline{\mathbb{F}}_{q}}):=\mathrm{rank}\,\mathrm{NS}(X_{\overline{\mathbb{F}}_{q}})$ equal to $20$ (so-called singular K3 surfaces) satisfy Parshin's conjecture.
\end{rem}

Our goal in this note is to compute the $K$-groups of a class of smooth projective surfaces over $\mathrm{Spec}\,\mathbb{F}_{q}$. Our strategy is to first compute the motivic cohomology $H^{i}(X,\mathbb{Z}(n))$ of surfaces $X$ which satisfy Parshin's conjecture and whose geometric \'{e}tale cohomology groups $H^{i}_{\mathrm{\acute{e}t}}(X_{\overline{\mathbb{F}}_{q}},\mathbb{Z}_{\ell})$ are free $\mathbb{Z}_{\ell}$-modules, $\ell\neq p$. This is achieved by using the work of Geisser-Levine \cite{GL01}, \cite{GL00}. We direct the reader to Theorem \ref{main} for the statement of our computation. By the list of cases in Remark \ref{Kahn on Parshin} where Parshin's conjecture is known, Theorem \ref{main} holds unconditionally for unirational surfaces, abelian surfaces, K3 surfaces of geometric Picard rank $\geq 20$ and hypersurfaces in $\mathbb{P}_{\mathbb{F}_{q}}^{3}$ of abelian-type (but of course it should hold in general!). Then we study the Atiyah-Hirzebruch spectral sequence
\begin{equation*}
E_{2}^{i,j}=H^{i-j}_{\mathcal{M}}(X,\mathbb{Z}(n))\Rightarrow K_{-i-j}(X)
\end{equation*}
to compute the $K$-theory of our class of surfaces. Our main result is the following:
\begin{thm}\label{intro theorem}
Let $X$ be a smooth projective geometrically irreducible surface over $\mathbb{F}_{q}$ with $H^{1}_{\mathrm{\acute{e}t}}(X_{\overline{\mathbb{F}}_{q}},\mathbb{Z}_{\ell})=H^{3}_{\mathrm{\acute{e}t}}(X_{\overline{\mathbb{F}}_{q}},\mathbb{Z}_{\ell})=0$ and $H^{2}_{\mathrm{\acute{e}t}}(X_{\overline{\mathbb{F}}_{q}},\mathbb{Z}_{\ell})$ a free $\mathbb{Z}_{\ell}$-module, for $\ell\neq p$. If Parshin's conjecture holds for $X$, then the motivic Atiyah-Hirzebruch spectral sequence
\begin{equation*}
E_{2}^{i,j}=H^{i-j}_{\mathcal{M}}(X,\mathbb{Z}(-j))\Rightarrow K_{-i-j}(X)
\end{equation*}
degenerates at the $E_{2}$-page. Moreover, the higher $K$-groups of $X$ are as follows:
\[ K_{n}(X)\simeq \begin{cases} 
      \left(\mathbb{Z}/(q-1)\mathbb{Z}\right)^{\oplus 2}\oplus\displaystyle\varinjlim_{r}H^{0}_{\mathrm{Zar}}(X,W_{r}\Omega_{X,\log}^{2})\oplus\bigoplus_{\ell\neq p}{H^{2}_{\mathrm{\acute{e}t}}(X_{\overline{\mathbb{F}}_{q}},\mathbb{Z}_{\ell}(2))}_{\Gamma} & \text{ if }n=1 \\
      0 & \text{ if } n=2m, m\geq 1 \\
      \left(\mathbb{Z}/(q^{m}-1)\mathbb{Z}\right)^{\oplus 2}\oplus\displaystyle\bigoplus_{\ell\neq p}{H^{2}_{\mathrm{\acute{e}t}}(X_{\overline{\mathbb{F}}_{q}},\mathbb{Z}_{\ell}(m+1))}_{\Gamma} & \text{ if }n=2m-1, m\geq 2\,. \\
   \end{cases}
\]
\end{thm} 
In particular, Theorem \ref{intro theorem} holds unconditionally for unirational surfaces, K3 surfaces of geometric Picard rank $\geq 20$ and hypersurfaces in $\mathbb{P}_{\mathbb{F}_{q}}^{3}$ of abelian-type. For rational surfaces, we recover (and slightly extend to encompass $p$-torsion too) the results of Coombes \cite{Coo87}.

In principle, the strategy of using \cite{GL01} and \cite{GL00} to compute the motivic cohomology of smooth projective surfaces over $\mathbb{F}_{q}$ for which Parshin's conjecture is known will work for any given surface for which one has a good handle over the arithmetic $\ell$-adic \'{e}tale cohomology and the logarithmic Hodge-Witt cohomology; one is not restricted to the class of especially nice surfaces singled out in Theorem \ref{intro theorem}. In this vein, we also include a computation of the motivic cohomology and $K$-theory of Enriques surfaces over $\mathbb{F}_{q}$ with $p=\mathrm{char}(\mathbb{F}_{q})>2$. These surfaces have $2$-torsion in their geometric \'{e}tale cohomology, and thus fall outside of the scope of Theorem \ref{intro theorem}. The result of our $K$-group calculation is the following:

\begin{thm}
Let $X$ be an Enriques surface over $\mathbb{F}_{q}$, with $\mathrm{char}(\mathbb{F}_{q})=p>2$. Then the $K$-groups $K_{n}(X)$ of $X$ are as follows:
\[ K_{n}(X)\simeq \begin{cases} 
      \mathbb{Z}^{2}\oplus\mathrm{Pic}(X) & \text{ if } n=0 \\
      \mathbb{Z}/2\mathbb{Z} & \text{ if }n=2m, m\geq 1 \\
      \left(\mathbb{Z}/(q^{m}-1)\mathbb{Z}\right)^{\oplus 2}\oplus\left(\mathrm{Pic}(X_{\overline{\mathbb{F}}_{q}})\otimes K_{2m-1}(\overline{\mathbb{F}}_{q})\right)^{\Gamma} & \text{ if }n=2m-1, m\geq 1\,.
   \end{cases}
\]
\end{thm}
  
This result was previously shown up to $2$- and $p$-torsion in \cite{Coo92} using a different method (Coombes shows that the Chow $[1/2]$-motive of an Enriques surface is closely related to the Chow $[1/2]$-motive of an associated rational surface, for which the $K$-groups were known by \cite{Coo87}). The absence of $p$-torsion was suspected, but the task of calculating the $2$-torsion was left as an open problem \cite[Remark 3.5]{Coo92}.

\begin{rem}
We often use Quillen's computation of $K_{n}(\mathbb{F}_{q})$, but for some of our statements it is enough to know that $K_{n}(\mathbb{F}_{q})$ is torsion. This latter fact is much easier to prove by noticing that $q^{i}-1$ kills the $i$-th Adams eigenspace $K_{n}(\mathbb{F}_{q})_{\mathbb{Q}}^{(i)}$. A generalisation of this argument appears in \cite{Sou84} and \cite[\S3]{Gei98}. 
\end{rem}

\begin{rem}
Finally, let us remark that a conjecture of Bass asserts that the $K$-groups of a regular scheme of finite type over $\mathrm{Spec}\,\mathbb{Z}$ are finitely generated. Taken together, the Bass and Parshin conjectures predict that the higher $K$-groups of a smooth projective variety over $\mathrm{Spec}\,\mathbb{F}_{q}$ should be finite. Of course, all of the higher $K$-groups (and motivic cohomology groups away from the $(2n,n)$ Chow diagonal) appearing in our calculations for various classes of surface are indeed finite.
\end{rem}

{\bf{Notation:}} We fix a prime number $p$ throughout. $\mathbb{F}_{q}$ denotes a finite field of characteristic $p$, and $\Gamma:=\mathrm{Gal}(\overline{\mathbb{F}}_{q}/\mathbb{F}_{q})$ is its absolute Galois group. For an abelian group $A$, we write $A_{\mathrm{tors}}$ for the torsion subgroup of $A$, and $A_{\ell\mathrm{-tors}}:=\varinjlim_{n}(\ker (A\xrightarrow{\ell^{n}}A))$ for the $\ell$-primary subgroup of $A$, for $\ell$ any prime. For $G$ a group acting on an abelian group $A$, we write $A^{G}$ for the $G$-invariants and $A_{G}$ for the $G$-coinvariants. 

 \section{The motivic cohomology of some surfaces over $\mathbb{F}_{q}$}\label{motivic cohomology section}

\begin{thm}\label{main}
Let $X$ be a smooth projective surface over $\mathbb{F}_{q}$ such that $H^{i}_{\mathrm{\acute{e}t}}(X_{\overline{\mathbb{F}}_{q}},\mathbb{Z}_{\ell})$ is a free $\mathbb{Z}_{\ell}$-module, for $\ell\neq p$, for each $i=0,1,2,3,4$. If $X$ satisfies Parshin's conjecture, then the motivic cohomology groups $H^{i}_{\mathcal{M}}(X,\mathbb{Z}(n))$ of $X$ are as in Table \ref{fig:table}. All of the groups occurring outside of the $(i,n)=(2n,n)$ diagonal are finite. In fact, for each prime $\ell$ (including $\ell=p$), let $|\,\,\,|_{\ell}:\mathbb{Q}_{\ell}\rightarrow\mathbb{Q}$ denotes the $\ell$-adic absolute value, normalised so that $|\ell|_{\ell}=\ell^{-1}$. For $\ell\neq p$, let 
\begin{equation*}
P_{i}(T):=\det(1-\mathrm{Frob}\cdot T\,|\,H^{i}_{\mathrm{\acute{e}t}}(X_{\overline{\mathbb{F}}_{q}},\mathbb{Q}_{\ell})).
\end{equation*}
By Deligne's proof of the Weil conjectures, $P_{i}(T)\in\mathbb{Z}[T]$ is a polynomial with integer coefficients, and is independent of the choice of prime $\ell\neq p$. Then for all $i\neq 2n$, $n\geq 2$, $(i,n)\neq (3,2)$ we have
\begin{equation*}
\#H^{i}_{\mathcal{M}}(X,\mathbb{Z}(n))=\prod_{\ell\neq p}|P_{i-1}(q^{n})|^{-1}_{\ell}=|P_{i-1}(q^{n})|_{p}\cdot|P_{i-1}(q^{n})|_{\infty}\,.
\end{equation*}
Finally, for $(i,n)=(3,2)$ the finite group $H^{3}_{\mathcal{M}}(X,\mathbb{Z}(2))$ has size
\begin{equation*}
\#H^{3}_{\mathcal{M}}(X,\mathbb{Z}(2))=|P_{2}(q^{2})|^{-1}_{p}
\cdot\prod_{\ell\neq p}|P_{2}(q^{2})|^{-1}_{\ell}=|P_{2}(q^{2})|_{\infty}
\end{equation*}
(Note of course that that $|P_{i-1}(q^{n})|_{\ell}=1$ for almost all primes $\ell$).
\end{thm}

\begin{rem}
Let $k$ be a field and let $X$ be a smooth connected variety over $k$. Then it is known (see \cite[Theorem 1.6]{Kah96}) that $H^{1}_{\mathcal{M}}(X,\mathbb{Z}(2))\simeq K_{3}(k)_{\mathrm{ind}}$ is the indecomposable part of $K_{3}(k)$ (recall that the Milnor $K$-group $K^{\mathrm{Mil}}_{3}(k)$ injects into $K_{3}(k)$ and $K_{3}(k)_{\mathrm{ind}}$ is defined to be the quotient of $K_{3}(k)$ by $K_{3}^{\mathrm{Mil}}(k)$). In the case that $k=\mathbb{F}_{q}$ is a finite field, $K_{3}^{\mathrm{Mil}}(\mathbb{F}_{q})=0$ and hence $H^{1}_{\mathcal{M}}(X,\mathbb{Z}(2))\simeq K_{3}(\mathbb{F}_{q})_{\mathrm{ind}}\simeq K_{3}(\mathbb{F}_{q})\simeq\mathbb{Z}/(q^{2}-1)\mathbb{Z}$.
\end{rem}

\begin{rem}
Recall that Lichtenbaum has defined Weil-\'{e}tale cohomology groups $H^{i}_{\mathrm{W}}(X,\mathbb{Z}(n))$ for smooth varieties $X$. These groups are conjectured to be finitely generated and related to special values of zeta functions (see \cite{Lic05} and \cite{Gei04}). By \cite[Theorem 7.1(a)]{Gei04} there is a long exact sequence
\begin{equation}\label{Weil-etale sequence}
\cdots\rightarrow H^{i}_{\mathrm{L}}(X,\mathbb{Z}(n))\rightarrow H^{i}_{\mathrm{W}}(X,\mathbb{Z}(n))\rightarrow H^{i-1}_{\mathrm{L}}(X,\mathbb{Q}(n))\rightarrow H^{i+1}_{\mathrm{L}}(X,\mathbb{Z}(n))\rightarrow\cdots
\end{equation}
where $H^{i}_{\mathrm{L}}(X,A(n)):=\mathbb{H}^{i}(X_{\mathrm{\acute{e}t}},A(n))$ denotes \'{e}tale motivic cohomology with coefficients in an abelian group $A$. Recall that rationally we have $H^{i}_{\mathrm{L}}(X,\mathbb{Q}(n))\simeq H^{i}_{\mathcal{M}}(X,\mathbb{Q}(n))$, which vanishes away from the $(i,n)=(2n,n)$ diagonal under Parshin's conjecture. In particular, the long exact sequence gives $H^{i}_{\mathrm{W}}(X,\mathbb{Z}(n))\simeq H^{i}_{\mathrm{L}}(X,\mathbb{Z}(n))$ for $i\neq 2n+1,2n+2$. Now recall that a consequence of the Beilinson-Lichtenbaum conjecture (itself a consequence of the Bloch-Kato conjecture/norm residue theorem of Rost-Voevodsky) is that $H^{i}_{\mathcal{M}}(X,\mathbb{Z}(n))\cong H^{i}_{\mathrm{L}}(X,\mathbb{Z}(n))$ for $i\leq n+1$. Altogether, we have  
\begin{equation*}
H^{i}_{\mathrm{W}}(X,\mathbb{Z}(n))\simeq H^{i}_{\mathcal{M}}(X,\mathbb{Z}(n))
\end{equation*}
for $i\leq n+1$ with $(i,n)\neq (1,0)$ for those $X$ satisfying Parshin's conjecture. For the class of surfaces considered in Theorem \ref{main} our computation together with \cite[Theorem 3.1]{Lic05} and \cite[\S7]{Gei04} also gives the Weil-\'{e}tale cohomology groups of $X$ apart from the two groups $H^{3}_{\mathrm{W}}(X,\mathbb{Z}(1))$ and $H^{4}_{\mathrm{W}}(X,\mathbb{Z}(1))$. These two undetermined groups $H^{3}_{\mathrm{W}}(X,\mathbb{Z}(1))$ and $H^{4}_{\mathrm{W}}(X,\mathbb{Z}(1))$ live in the following exact sequence (coming from \eqref{Weil-etale sequence} together with $H^{3}_{\mathrm{L}}(X,\mathbb{Z}(1))\simeq\mathrm{Br}(X)$ and $H^{2}_{\mathcal{M}}(X,\mathbb{Z}(1))\simeq\mathrm{Pic}(X)$):
\begin{equation*}
0\rightarrow\mathrm{Br}(X)\rightarrow H^{3}_{\mathrm{W}}(X,\mathbb{Z}(1))\rightarrow\mathrm{NS}(X)\otimes\mathbb{Q}\rightarrow H^{3}_{\mathrm{\acute{e}t}}(X,\mathbb{Q}/\mathbb{Z}(1))\rightarrow H^{4}_{\mathrm{W}}(X,\mathbb{Z}(1))\rightarrow 0\,.
\end{equation*}
The Lichtenbaum Chow group $\mathrm{CH}^{2}_{\mathrm{L}}(X):=H^{4}_{\mathrm{L}}(X,\mathbb{Z}(2))$ has $\mathrm{CH}^{2}_{\mathrm{L}}(X)_{\mathrm{tors}}\simeq H^{3}_{\mathrm{\acute{e}t}}(X,\mathbb{Q}/\mathbb{Z}(2))$.
\end{rem}

\begin{lemma}\label{torsion in motivic cohomology lemma}
Let $X$ be a smooth separated scheme over $\mathbb{F}_{q}$. Let $i$ and $n$ be integers. If $H^{i-1}_{\mathcal{M}}(X,\mathbb{Z}(n))$ is torsion then $H^{i}_{\mathcal{M}}(X,\mathbb{Z}(n))_{\mathrm{tors}}\simeq H^{i-1}_{\mathcal{M}}(X,\mathbb{Q}/\mathbb{Z}(n))$.
\end{lemma}
\begin{proof}
By hypothesis, the first term in the short exact sequence
\begin{equation*}
0\rightarrow H^{i-1}_{\mathcal{M}}(X,\mathbb{Z}(n))\otimes_{\mathbb{Z}}\mathbb{Q}/\mathbb{Z}\rightarrow H^{i-1}_{\mathcal{M}}(X,\mathbb{Q}/\mathbb{Z}(n))\rightarrow H^{i}_{\mathcal{M}}(X,\mathbb{Z}(n))_{\mathrm{tors}}\rightarrow 0
\end{equation*}
is a torsion abelian group tensored with a divisible group, and therefore vanishes.
\end{proof}

\begin{lemma}\label{GL}
Let $X$ be a smooth separated surface over $\mathbb{F}_{q}$. Let $i$ and $n$ be integers. If $n\neq 2$, or if $n=2$ and $i\neq 4, 5$, then we have 
\begin{equation*}
H^{i-1}_{\mathcal{M}}(X,\mathbb{Q}/\mathbb{Z}(n))\simeq\varinjlim_{r}H_{\mathrm{Zar}}^{i-n-1}(X,W_{r}\Omega^{n}_{X,\log})\oplus\bigoplus_{\ell\neq p}H^{i-1}_{\mathrm{\acute{e}t}}(X,\mathbb{Q}_{\ell}/\mathbb{Z}_{\ell}(n))
\end{equation*}
(where $p=\mathrm{char}(\mathbb{F}_{q})$).
\end{lemma}
\begin{proof}
For $n=0,1$ this is clear. For $i\geq 6$ the groups vanish because $X$ is a surface. In general, for the $p$-primary torsion summand we have 
\begin{equation*}
H^{i-1}_{\mathcal{M}}(X,\mathbb{Q}_{p}/\mathbb{Z}_{p}(n))\simeq\varinjlim_{r}H_{\mathrm{Zar}}^{i-n-1}(X,W_{r}\Omega^{n}_{X,\log})
\end{equation*}
by \cite[Theorem 8.4]{GL00}. Here $W_{r}\Omega^{n}_{X,\log}$ is the (abuse of) notation for $\epsilon_{\ast}W_{r}\Omega_{X,\log}^{n}$, where $\epsilon:X_{\mathrm{\acute{e}t}}\rightarrow X_{\mathrm{Zar}}$ is the change-of-topology map and $W_{r}\Omega_{X,\log}^{n}$ is the logarithmic Hodge-Witt sheaf on $X_{\mathrm{\acute{e}t}}$; it is the subsheaf of $W_{r}\Omega_{X}^{n}$ \'{e}tale locally generated by sections of the form $d\log[x_{1}]_{r}\ldots d\log[x_{n}]_{r}$, where $x_{i}\in\mathcal{O}_{X}^{\ast}$ and $[x_{i}]_{r}\in W_{r}\mathcal{O}_{X}$ is the Teichm\"{u}ller lift of $x_{i}$. By \cite[Theorem 1.1]{GL01}, the Bloch-Kato conjecture (the norm residue theorem of Rost-Voevodsky) implies that the $\ell$-primary torsion is
\begin{equation*}
H^{i-1}_{\mathcal{M}}(X,\mathbb{Q}_{\ell}/\mathbb{Z}_{\ell}(n))\simeq H^{i-1}_{\mathrm{Zar}}(X,\tau_{\leq n}R\epsilon_{\ast}\mathbb{Q}_{\ell}/\mathbb{Z}_{\ell}(n))
\end{equation*}
where $\epsilon\,:\,X_{\mathrm{\acute{e}t}}\rightarrow X_{\mathrm{Zar}}$ is the change of topology map. Since the $\ell$-cohomological dimension of an affine scheme of finite type over $\mathrm{Spec}\,\mathbb{F}_{q}$ of dimension $2$ is $\leq 3$ \cite[Tag 0F0V]{Stacks}, we have
\begin{equation*}
H^{i-1}_{\mathrm{Zar}}(X,\tau_{\leq n}R\epsilon_{\ast}\mathbb{Q}_{\ell}/\mathbb{Z}_{\ell}(n))\simeq H^{i-1}_{\mathrm{\acute{e}t}}(X,\mathbb{Q}_{\ell}/\mathbb{Z}_{\ell}(n))
\end{equation*}
for all $i$ whenever $n\geq 3$. Finally let us treat the case $n=2$. By \cite[Theorem 1.2]{GL01} we have that the cycle class map
\begin{equation*}
H^{i-1}_{\mathcal{M}}(X,\mathbb{Q}_{\ell}/\mathbb{Z}_{\ell}(2))\rightarrow  H^{i-1}_{\mathrm{\acute{e}t}}(X,\mathbb{Q}_{\ell}/\mathbb{Z}_{\ell}(2))
\end{equation*}
is an isomorphism for $i\leq 3$. 
\end{proof}

\begin{rem}
In the case $(i,n)=(4,2)$, the cycle class map 
\begin{equation*}
H^{3}_{\mathcal{M}}(X,\mathbb{Q}_{\ell}/\mathbb{Z}_{\ell}(2))\rightarrow  H^{3}_{\mathrm{\acute{e}t}}(X,\mathbb{Q}_{\ell}/\mathbb{Z}_{\ell}(2))
\end{equation*}
is injective by \cite[Theorem 1.1]{GL01}, but is not surjective in general. Indeed, there is a short exact sequence
\begin{equation*}
0\rightarrow H^{1}_{\mathrm{Zar}}(X,\mathcal{K}_{2})\otimes_{\mathbb{Z}}\mathbb{Q}_{\ell}/\mathbb{Z}_{\ell}\rightarrow N^{1}H^{3}_{\mathrm{\acute{e}t}}(X,\mathbb{Q}_{\ell}/\mathbb{Z}_{\ell}(2))\rightarrow\mathrm{CH}^{2}(X)_{\ell\mathrm{-tors}}\rightarrow 0
\end{equation*}
where the middle term denotes the first step of the coniveau filtration (see e.g. \cite[\S1]{Sai91} for a nice summary). The Weil conjectures imply that $H^{3}_{\mathrm{\acute{e}t}}(X,\mathbb{Q}_{\ell}/\mathbb{Z}_{\ell}(2))$ is a finite group, so $H^{1}_{\mathrm{Zar}}(X,\mathcal{K}_{2})\otimes_{\mathbb{Z}}\mathbb{Q}_{\ell}/\mathbb{Z}_{\ell}$ is finite and divisible, hence trivial. Thus 
\begin{equation*}
\mathrm{CH}^{2}(X)_{\ell-\mathrm{tors}}\simeq N^{1}H^{3}_{\mathrm{\acute{e}t}}(X,\mathbb{Q}_{\ell}/\mathbb{Z}_{\ell}(2))\,.
\end{equation*}
If $H^{2}_{\mathcal{M}}(X,\mathbb{Z}(2))$ is torsion (as predicted by e.g. Parshin's conjecture), then the $\ell$-primary part of the exact sequence in Lemma \ref{torsion in motivic cohomology lemma} shows that $\mathrm{CH}^{2}(X)_{\ell\mathrm{-tors}}\simeq H^{4}_\mathcal{M}(X,\mathbb{Z}(2))_{\ell\mathrm{-tors}}\simeq H^{3}_{\mathcal{M}}(X,\mathbb{Q}_{\ell}/\mathbb{Z}_{\ell}(2))$. Consequently, if $H^{2}_{\mathcal{M}}(X,\mathbb{Z}(2))$ is torsion, then 
\begin{equation*}
H^{3}_{\mathcal{M}}(X,\mathbb{Q}_{\ell}/\mathbb{Z}_{\ell}(2))\simeq N^{1}H^{3}_{\mathrm{\acute{e}t}}(X,\mathbb{Q}_{\ell}/\mathbb{Z}_{\ell}(2))\,. 
\end{equation*}
\end{rem}

\begin{flushleft}\emph{Proof of 
Theorem \ref{main}.} The rows for $n=0,1$ are standard \cite[Corollary 4.2]{MVW06}. The entries along the line $i=2n$ are by the comparison with Chow groups \cite[Corollary 19.2]{MVW06}. Recall also that $H^{i}_{\mathcal{M}}(X,\mathbb{Z}(n))=0$ for $i>n+2$ and for $i>2n$ \cite[Theorem 3.6, Theorem 19.3]{MVW06}. Bloch's formula \cite{Blo86} says that $\mathrm{CH}^{n}(X,m)\simeq H^{2n-m}_{\mathcal{M}}(X,\mathbb{Z}(n))$ rationally agrees with the weight $n$ graded piece of $K$-theory:
\begin{equation*}
K_{m}(X)^{(n)}_{\mathbb{Q}}\simeq H^{2n-m}_{\mathcal{M}}(X,\mathbb{Q}(n))\,.
\end{equation*}
In particular, Parshin's conjecture for $X$ is equivalent to the statement that $H^{j}_{\mathcal{M}}(X,\mathbb{Z}(n))$ is torsion for all $j\neq 2n$. Suppose that $i-1\neq 2n$. By Lemma \ref{torsion in motivic cohomology lemma} we have \end{flushleft}
\begin{equation*}
H^{i}_{\mathcal{M}}(X,\mathbb{Z}(n))=H^{i}_{\mathcal{M}}(X,\mathbb{Z}(n))_{\mathrm{tors}}\simeq H^{i-1}_{\mathcal{M}}(X,\mathbb{Q}/\mathbb{Z}(n))\,.
\end{equation*}
Lemma \ref{GL} then gives that
\begin{equation*}
H^{i}_{\mathcal{M}}(X,\mathbb{Z}(n))\simeq \varinjlim_{r}H_{\mathrm{Zar}}^{i-n-1}(X,W_{r}\Omega^{n}_{X,\log})\oplus\bigoplus_{\ell\neq p}H^{i-1}_{\mathrm{\acute{e}t}}(X,\mathbb{Q}_{\ell}/\mathbb{Z}_{\ell}(n))
\end{equation*}
for all $(i,n)\neq (5,2)$. For $n\geq 3$ or $i\leq n$ we have $H_{\mathrm{Zar}}^{i-n-1}(X,W_{r}\Omega^{n}_{X,\log})=0$ and hence $H^{i}_{\mathcal{M}}(X,\mathbb{Z}(n))$ has no $p$-torsion in this range. We also see that $H^{0}_{\mathcal{M}}(X,\mathbb{Z}(n))$ has no $\ell$-torsion. Now consider the short exact sequence
\begin{equation*}
0\rightarrow H^{i-1}_{\mathrm{\acute{e}t}}(X,\mathbb{Z}_{\ell}(n))\otimes_{\mathbb{Z}}\mathbb{Q}_{\ell}/\mathbb{Z}_{\ell}\rightarrow H^{i-1}_{\mathrm{\acute{e}t}}(X,\mathbb{Q}_{\ell}/\mathbb{Z}_{\ell}(n))\rightarrow H^{i}_{\mathrm{\acute{e}t}}(X,\mathbb{Z}_{\ell}(n))_{\mathrm{tors}}\rightarrow 0\,.
\end{equation*} 
Since $i-1\neq 2n$, the Weil conjectures imply that $H^{i-1}_{\mathrm{\acute{e}t}}(X,\mathbb{Z}_{\ell}(n))$ and $H^{i}_{\mathrm{\acute{e}t}}(X,\mathbb{Z}_{\ell}(n))$ are torsion \cite[p. 781]{CTSS83}. Thus the first term vanishes and we conclude that the $\ell$-primary torsion summand of $H^{i}_{\mathcal{M}}(X,\mathbb{Z}(n))$ is 
\begin{equation}\label{l-torsion}
H^{i}_{\mathcal{M}}(X,\mathbb{Z}(n))_{\ell\mathrm{-tors}}\simeq H^{i-1}_{\mathrm{\acute{e}t}}(X,\mathbb{Q}_{\ell}/\mathbb{Z}_{\ell}(n))\simeq H^{i}_{\mathrm{\acute{e}t}}(X,\mathbb{Z}_{\ell}(n))\,.
\end{equation}
Since the $\mathbb{F}_{q}$ has cohomological dimension $1$, the Hochschild-Serre spectral sequence
\begin{equation*}
E_{2}^{i,j}=H^{i}(\Gamma,H^{j}_{\mathrm{\acute{e}t}}(X_{\overline{\mathbb{F}}_{q}},\mathbb{Z}_{\ell}(n)))\Rightarrow H^{i+j}_{\mathrm{\acute{e}t}}(X,\mathbb{Z}_{\ell}(n))
\end{equation*}
gives short exact sequences
\begin{equation}\label{HS sequence}
0\rightarrow {H^{i-1}_{\mathrm{\acute{e}t}}(X_{\overline{\mathbb{F}}_{q}},\mathbb{Z}_{\ell}(n))}_{\Gamma}\rightarrow H^{i}_{\mathrm{\acute{e}t}}(X,\mathbb{Z}_{\ell}(n))\rightarrow H^{i}_{\mathrm{\acute{e}t}}(X_{\overline{\mathbb{F}}_{q}},\mathbb{Z}_{\ell}(n))^{\Gamma}\rightarrow 0\,.
\end{equation}
We see then that $H^{i}_{\mathrm{\acute{e}t}}(X,\mathbb{Z}_{\ell}(n))=0$ for $i\geq 6$ because $X_{\overline{\mathbb{F}}_{q}}$ is a surface, and hence $H^{i}_{\mathcal{M}}(X,\mathbb{Z}(n))_{\ell\mathrm{-tors}}=0$ for $i\geq 6$ by \eqref{l-torsion}. By the Weil conjectures, the groups $H^{i}_{\mathrm{\acute{e}t}}(X_{\overline{\mathbb{F}}_{q}},\mathbb{Z}_{\ell}(n))^{\Gamma}$ are torsion for $i\neq 2n$ \cite[p. 781]{CTSS83}. But by hypothesis, $H^{i}_{\mathrm{\acute{e}t}}(X_{\overline{\mathbb{F}}_{q}},\mathbb{Z}_{\ell})$ and hence $H^{i}_{\mathrm{\acute{e}t}}(X_{\overline{\mathbb{F}}_{q}},\mathbb{Z}_{\ell}(n))^{\Gamma}$ are torsion-free for all $(i,n)$. So we conclude that $H^{i}_{\mathrm{\acute{e}t}}(X_{\overline{\mathbb{F}}_{q}},\mathbb{Z}_{\ell}(n))^{\Gamma}=0$ for all $i\neq 2n$, and hence \eqref{HS sequence} and \eqref{l-torsion} give
\begin{equation*}
H^{i}_{\mathcal{M}}(X,\mathbb{Z}(n))_{\ell\mathrm{-tors}}\simeq {H^{i-1}_{\mathrm{\acute{e}t}}(X_{\overline{\mathbb{F}}_{q}},\mathbb{Z}_{\ell}(n))}_{\Gamma}
\end{equation*} 
for all $i-1\neq 2n,2n-1$. Since $H^{0}_{\mathrm{\acute{e}t}}(X_{\overline{\mathbb{F}}_{q}},\mathbb{Z}_{\ell}(n))\simeq\mathbb{Z}_{\ell}(n)$ and $H^{4}_{\mathrm{\acute{e}t}}(X_{\overline{\mathbb{F}}_{q}},\mathbb{Z}_{\ell}(n))\simeq\mathbb{Z}_{\ell}(n-2)$, we get $H^{1}_{\mathcal{M}}(X,\mathbb{Z}(n))_{\ell\mathrm{-tors}}\simeq\mathbb{Z}_{\ell}(n)_{\Gamma}\simeq\mathbb{Z}_{\ell}/(q^{n}-1)\mathbb{Z}_{\ell}$ for all $n\geq 0$ and $H^{5}_{\mathcal{M}}(X,\mathbb{Z}(n))_{\ell\mathrm{-tors}}\simeq\mathbb{Z}_{\ell}(n-2)_{\Gamma}\simeq\mathbb{Z}_{\ell}/(q^{n-2}-1)\mathbb{Z}_{\ell}$ for all $n\geq 3$. Summing over $\ell$-primary summands and using the Chinese Remainder Theorem gives $H^{1}_{\mathcal{M}}(X,\mathbb{Z}(n))\simeq\mathbb{Z}/(q^{n}-1)\mathbb{Z}$ for $n\geq 0$ and $H^{5}_{\mathcal{M}}(X,\mathbb{Z}(n))\simeq\mathbb{Z}/(q^{n-2}-1)\mathbb{Z}$ for $n\geq 3$. Note that ${H^{i-1}_{\mathrm{\acute{e}t}}(X_{\overline{\mathbb{F}}_{q}},\mathbb{Z}_{\ell}(n))}_{\Gamma}$ is a finite group for $i-1\neq n$ by the Weil conjectures (see e.g. \cite[p. 781]{CTSS83}), and the size of ${H^{i-1}_{\mathrm{\acute{e}t}}(X_{\overline{\mathbb{F}}_{q}},\mathbb{Z}_{\ell}(n))}_{\Gamma}$ is $|P_{i-1}(q^{n})|_{\ell}$ (see e.g. \cite[Lemma 3.2]{BN78} or \cite[p. 782]{CTSS83}). 

Finally let us treat the $p$-primary torsion component of $H^{3}_{\mathcal{M}}(X,\mathbb{Z}(2))$. This discussion is well-known - for a very nice write-up see \cite[Appendix]{KY18}. We give a brief summary. As we have seen, $H^{3}_{\mathcal{M}}(X,\mathbb{Z}(2))_{p\mathrm{-tors}}\simeq\varinjlim_{r}H^{0}_{\mathrm{Zar}}(X,W_{r}\Omega^{2}_{X,\log})$ where the direct limit is over the maps $\underline{p}:W_{r}\Omega_{X}^{2}\rightarrow W_{r+1}\Omega^{2}_{X}$ \cite[Ch. I, Proposition 3.4]{Ill79}. Note that each $W_{r}\Omega_{X}^{n}$ is a quasi-coherent $W_{r}\mathcal{O}_{X}$-module \cite[Ch. I Proposition 1.13.1]{Ill79}, so the Zariski cohomology of $W_{r}\Omega_{X}^{n}$ agrees with the \'{e}tale cohomology of the associated sheaf on $X_{\mathrm{\acute{e}t}}$.  Also note that for all $n\geq 0$, the map $H^{0}_{\mathrm{Zar}}(X,W_{r}\Omega^{n}_{X,\log})\rightarrow H^{0}_{\mathrm{\acute{e}t}}(X,W_{r}\Omega^{n}_{X,\log})$ induced by $\epsilon:X_{\mathrm{\acute{e}t}}\rightarrow X_{\mathrm{Zar}}$ is an isomorphism. In particular, we may work with the \'{e}tale cohomology groups of the $W_{r}\Omega^{n}_{X,\log}$ in what follows. Notice that if $x$ is an \'{e}tale local section of $W_{r}\Omega^{2}_{X,\log}$ then there exists an \'{e}tale local section $y$ of $W_{r+1}\Omega_{X,\log}^{2}$ with $x=Ry=Fy$. Hence $Vx=VFy=py=\underline{p}x$ and we see that $\varinjlim_{r}H^{0}_{\mathrm{\acute{e}t}}(X,W_{r}\Omega^{2}_{X,\log})\simeq\varinjlim_{V}H^{0}_{\mathrm{\acute{e}t}}(X,W_{r}\Omega^{2}_{X,\log})$. The Weil conjectures imply that $H^{3}_{\mathcal{M}}(X,\mathbb{Z}(2))_{p\mathrm{-tors}}\simeq\varinjlim_{r}H^{0}_{\mathrm{\acute{e}t}}(X,W_{r}\Omega^{2}_{X,\log})$ is a finite group \cite[Corollaire 4.23]{GS88}. By \cite[Corollary 4.2.2]{Ill83}, for each $r\geq 1$ the multiplication map $W_{r}\mathcal{O}_{X}\times W_{r}\Omega_{X}^{2}\rightarrow W_{r}\Omega_{X}^{2}$ induces a perfect $W_{r}(\mathbb{F}_{q})$-bilinear pairing 
\begin{equation*}
H_{\mathrm{\acute{e}t}}^{2}(X,W_{r}\mathcal{O}_{X})\times H_{\mathrm{\acute{e}t}}^{0}(X,W_{r}\Omega_{X}^{2})\xrightarrow{\cup} H_{\mathrm{\acute{e}t}}^{2}(X,W_{r}\Omega_{X}^{2})\xrightarrow{\sim}W_{r}(\mathbb{F}_{q})\,.
\end{equation*}
Since $xVy=V(Fx.y)$, taking limits gives a perfect pairing 
\begin{equation*}
\varprojlim_{F}H_{\mathrm{\acute{e}t}}^{2}(X,W_{r}\mathcal{O}_{X})\times \varinjlim_{V}H_{\mathrm{\acute{e}t}}^{0}(X,W_{r}\Omega_{X}^{2})\xrightarrow{\cup}\varinjlim_{V}W_{r}(\mathbb{F}_{q})\simeq W(\mathbb{F}_{q})[1/p]/W(\mathbb{F}_{q})\,.
\end{equation*}
Now, by \cite[A.3]{KY18} there is a short exact sequence\begin{equation}\label{KY sequence}
0\rightarrow\varinjlim_{V}W_{r}\Omega_{X,\log}^{2}\rightarrow\varinjlim_{V}W_{r}\Omega_{X}^{2}\xrightarrow{1-F'}\varinjlim_{V}W_{r}\Omega_{X}^{2}\rightarrow 0
\end{equation}
of sheaves on $X_{\mathrm{\acute{e}t}}$. Here the modified Frobenius operator $F'$ is defined as follows: the Frobenius and restriction homomorphisms $F,R:W_{r+1}\Omega_{X}^{2}\rightarrow W_{r}\Omega_{X}^{2}$ factor through the surjection $W_{r+1}\Omega_{X}^{2}\twoheadrightarrow W_{r+1}\Omega_{X}^{2}/V^{n}\Omega_{X}^{2}$. Let $\tilde{F},\tilde{R}:W_{r+1}\Omega_{X}^{2}/V^{n}\Omega_{X}^{2}\twoheadrightarrow W_{r}\Omega_{X}^{2}$ be the induced (surjective) homomorphisms. By \cite[A.7]{KY18} the maps $\tilde{R}$ induce an isomorphism $\varinjlim_{V}(W_{r+1}\Omega_{X}^{2}/V^{n}\Omega_{X}^{2})\simeq\varinjlim_{V}W_{r}\Omega_{X}^{2}$, and the map $F'$ appearing in \eqref{KY sequence} is the inductive limit of the $\tilde{F}$ under this isomorphism. Now, the map $F:W_{r+1}\mathcal{O}_{X}\rightarrow W_{r}\mathcal{O}_{X}$ factors as $W_{r+1}\mathcal{O}_{X}\xrightarrow{\sigma}W_{r+1}\mathcal{O}_{X}\xrightarrow{R}W_{r}\mathcal{O}_{X}$ where $\sigma$ is the Witt vector Frobenius, and since $H^{2}_{\mathrm{\acute{e}t}}(X,W\mathcal{O}_{X})\simeq\varprojlim_{R}H^{2}_{\mathrm{\acute{e}t}}(X,W_{r}\mathcal{O}_{X})$ we find that $\varprojlim_{F}H_{\mathrm{\acute{e}t}}^{2}(X,W_{r}\mathcal{O}_{X})\simeq 
\varprojlim_{\sigma}H_{\mathrm{\acute{e}t}}^{2}(X,W\mathcal{O}_{X})$. Thus we have shown that $\varinjlim_{V}H_{\mathrm{\acute{e}t}}^{0}(X,W_{r}\Omega_{X}^{2})$ is the Pontryagin dual of $\varprojlim_{\sigma}H_{\mathrm{\acute{e}t}}^{2}(X,W\mathcal{O}_{X})$. Under the isomorphism $\varprojlim_{F}H_{\mathrm{\acute{e}t}}^{2}(X,W_{r}\mathcal{O}_{X})\simeq 
\varprojlim_{\sigma}H_{\mathrm{\acute{e}t}}^{2}(X,W\mathcal{O}_{X})$, the Pontryagin dual of the map $1-F'$ corresponds to the map sending a homomorphism $f:\varprojlim_{\sigma}H_{\mathrm{\acute{e}t}}^{2}(X,W\mathcal{O}_{X})\rightarrow W(\mathbb{F}_{q})[1/p]/W(\mathbb{F}_{q})$ to $f-\sigma\circ f\circ\sigma^{-1}$ (where $\sigma:\varprojlim_{\sigma}H_{\mathrm{\acute{e}t}}^{2}(X,W\mathcal{O}_{X})\rightarrow\varprojlim_{\sigma}H_{\mathrm{\acute{e}t}}^{2}(X,W\mathcal{O}_{X})$ is the endomorphism induced by $\sigma$ on $H_{\mathrm{\acute{e}t}}^{2}(X,W\mathcal{O}_{X})$). Overall, we have shown that 
\begin{equation*}
H^{3}_{\mathcal{M}}(X,\mathbb{Z}(2))_{p-\mathrm{tors}}\simeq\mathrm{Hom}_{W(\mathbb{F}_{q})}\left(\left[\varprojlim_{\sigma}H^{2}_{\mathrm{\acute{e}t}}(X,W\mathcal{O}_{X})\right]_{\sigma},W(\mathbb{F}_{q})[1/p]/W(\mathbb{F}_{q})\right)
\end{equation*}
where 
\begin{equation*}
\left[\varprojlim_{\sigma}H^{2}_{\mathrm{\acute{e}t}}(X,W\mathcal{O}_{X})\right]_{\sigma}:=\mathrm{coker}\left(\varprojlim_{\sigma}H^{2}_{\mathrm{\acute{e}t}}(X,W\mathcal{O}_{X})\xrightarrow{1-\sigma}\varprojlim_{\sigma}H^{2}_{\mathrm{\acute{e}t}}(X,W\mathcal{O}_{X})\right)\,.
\end{equation*}
We claim that $\left[\varprojlim_{\sigma}H^{2}_{\mathrm{\acute{e}t}}(X,W\mathcal{O}_{X})\right]_{\sigma}$ is finite. Indeed, it is a torsion group because $(\varprojlim_{\sigma}H^{2}_{\mathrm{\acute{e}t}}(X,W\mathcal{O}_{X}))\otimes_{\mathbb{Z}_{p}}\mathbb{Q}_{p}$ is isomorphic to the slope-zero part of $H^{2}_{\mathrm{cris}}(X/W(\mathbb{F}_{q}))\otimes_{\mathbb{Z}_{p}}\mathbb{Q}_{p}$, and hence $\left[\varprojlim_{\sigma}H^{2}_{\mathrm{\acute{e}t}}(X,W\mathcal{O}_{X})\right]_{\sigma}\otimes_{\mathbb{Z}_{p}}\mathbb{Q}_{p}=0$ (taking the cokernel of $1-\sigma$ commutes with tensoring with $\mathbb{Q}_{p}$ because $\mathbb{Q}_{p}$ is flat over $\mathbb{Z}_{p}$). But the torsion in $\varprojlim_{\sigma}H^{2}_{\mathrm{\acute{e}t}}(X,W\mathcal{O}_{X})$ injects into $H^{2}_{\mathrm{\acute{e}t}}(X,W\mathcal{O}_{X})_{V-\mathrm{tors}}$. By \cite[Ch. II Remarque 6.4]{Ill79}, $H^{2}_{\mathrm{\acute{e}t}}(X,W\mathcal{O}_{X})_{V-\mathrm{tors}}$ is canonically identified with the Pontryagin dual of the contravariant Dieudonn\'{e} module of $\mathrm{Pic}^{\wedge}_{X/\mathbb{F}_{q}}/\mathrm{Pic}^{\wedge}_{X/\mathbb{F}_{q},\mathrm{red}}$, and is thus a finite group. Finally, by Dieudonn\'{e} theory, we have that (the Pontryagin dual of) $\left[\varprojlim_{\sigma}H^{2}_{\mathrm{\acute{e}t}}(X,W\mathcal{O}_{X})_{\mathrm{tors}}\right]_{\sigma}$ is isomorphic to $\mathrm{Hom}_{\mathrm{grp-sch}/\mathbb{F}_{q}}(\mathrm{Pic}^{0}_{X/\mathbb{F}_{q}},\mathbb{G}_{m})$, the group of group scheme homomorphisms $\mathrm{Pic}^{0}_{X/\mathbb{F}_{q}}\rightarrow\mathbb{G}_{m}$ over $\mathbb{F}_{q}$. Thus
\begin{equation}
\#H^{3}_{\mathcal{M}}(X,\mathbb{Z}(2))_{p-\mathrm{tors}}=\#\mathrm{Hom}_{\mathrm{grp-sch}/\mathbb{F}_{q}}(\mathrm{Pic}^{0}_{X/\mathbb{F}_{q}},\mathbb{G}_{m})\cdot |P_{2}(q^{2})|^{-1}_{p}
\end{equation}
where we have used that the characteristic polynomial of Frobenius on $H^{2}_{\mathrm{cris}}(X/W(\mathbb{F}_{q}))$ is equal to $P_{2}(T)$ by \cite{KM74}. But $\mathrm{Pic}^{0}_{X/\mathbb{F}_{q}}$ is projective under our hypotheses, so the only scheme morphisms $\mathrm{Pic}^{0}_{X/\mathbb{F}_{q}}\rightarrow\mathbb{G}_{m}$ are constant morphisms, and thus there is only one group scheme homomorphism $\mathrm{Pic}^{0}_{X/\mathbb{F}_{q}}\rightarrow\mathbb{G}_{m}$. 
\begin{flushright}
$\square$
\end{flushright}

\begin{rem}\label{no p-torsion}
If $H^{0}(X,\Omega^{2}_{X})=0$ then $H^{0}(X,W_{r}\Omega_{X,\log}^{2})=0$. Indeed, induction using $W_{r}\Omega^{2}_{X,\log}\simeq W_{\bigcdot}\Omega^{2}_{X,\log}\otimes^{L}\mathbb{Z}/p^{r}\mathbb{Z}$ \cite[Ch. I Corollaire 5.7.5]{Ill79} reduces us to showing that $H^{0}(X,\Omega_{X,\log}^{2})=0$, where $\Omega^{2}_{X,\log}:=W_{1}\Omega^{2}_{X,\log}$. But $H^{0}(X,\Omega_{X,\log}^{2})\hookrightarrow H^{0}(X,Z\Omega_{X}^{2})=H^{0}(X,\Omega_{X}^{2})$ by \cite[Ch. 0 Remarque 2.5.2]{Ill79}.
\end{rem}

Recall that a smooth projective variety $X$ over an algebraically closed field $k$ of characteristic $p>0$ is called Shioda-supersingular \cite{Shi74} if its Picard number $\rho(X):=\mathrm{rank}\,\mathrm{NS}(X)$ is equal to $b_{2}:=\mathrm{dim}\, H^{2}(X,\mathbb{Q}_{\ell})$. For example, any unirational surface is Shioda-supersingular \cite[Corollary 2]{Shi74}. More generally, any smooth projective rationally chain connected variety is Shioda-supersingular \cite[Theorem 1.2]{GJ18}. We shall say that a smooth projective variety $X$ over $\mathbb{F}_{q}$ is Shioda-supersingular if $X_{\overline{\mathbb{F}}_{q}}$ is Shioda-supersingular.

\begin{rem}\label{Shioda vs super}
Recall that a surface $X$ is called supersingular if the Newton polygon of the $F$-crystal $H^{2}_{\mathrm{cris}}(X/W(k))$ is isoclinic. Being Shioda-supersingular is equivalent to $\mathrm{NS}(X)\otimes_{\mathbb{Z}}W(k)[1/p]\rightarrow H^{2}_{\mathrm{cris}}(X/W(k))_{\mathbb{Q}}$ being a bijection, and hence the Frobenius on $H^{2}_{\mathrm{cris}}(X/W(k))$ acts as multiplication-by-$p$ for $X$ Shioda-supersingular. Therefore Shioda-supersingular surfaces are supersingular. The converse is true if $X$ satisfies the Tate conjecture for divisors together with the variational Tate conjecture for divisors (proven by Morrow \cite[Theorem 0.2]{Mor19}).
\end{rem}

\begin{cor}\label{supersingular}
Let $X$ be a smooth projective Shioda-supersingular surface over $\mathbb{F}_{q}$ such that $H^{3}_{\mathrm{\acute{e}t}}(X_{\overline{\mathbb{F}}_{q}},\mathbb{Z}_{\ell})$ is a free $\mathbb{Z}_{\ell}$-module. Let $n\geq 3$ and suppose that $H^{2}_{\mathcal{M}}(X,\mathbb{Z}(n))$ and $H^{3}_{\mathcal{M}}(X,\mathbb{Z}(n))$ are torsion (e.g. if $X$ satisfies Parshin's conjecture). Then 
\begin{equation*}
H^{3}_{\mathcal{M}}(X,\mathbb{Z}(n))\simeq\bigoplus_{\ell\neq p}{(\mathrm{NS}(X_{\overline{\mathbb{F}}_{q}})\otimes\mathbb{Z}_{\ell})(n-1)}_{\Gamma}\,.
\end{equation*}
When $n=2$ we have
\begin{equation*}
H^{3}_{\mathcal{M}}(X,\mathbb{Z}(2))\simeq\displaystyle\varinjlim_{r}H^{0}_{\mathrm{Zar}}(X,W_{r}\Omega_{X,\log}^{2})\oplus\bigoplus_{\ell\neq p}{(\mathrm{NS}(X_{\overline{\mathbb{F}}_{q}})\otimes\mathbb{Z}_{\ell})(1)}_{\Gamma}\,.
\end{equation*}
\end{cor}
\begin{proof}
Arguing just as in the proof of Theorem \ref{main}, we have $H^{3}_{\mathcal{M}}(X,\mathbb{Z}(n))_{\ell\mathrm{-tors}}\simeq {H^{2}_{\mathrm{\acute{e}t}}(X_{\overline{\mathbb{F}}_{q}},\mathbb{Z}_{\ell}(n))}_{\Gamma}$ for $n\geq 2$, and $H^{3}_{\mathcal{M}}(X,\mathbb{Z}(n))_{p\mathrm{-tors}}=0$ for $n\geq 3$. Now observe that the injective map 
\begin{equation*}
\mathrm{NS}(X_{\overline{\mathbb{F}}_{q}})\otimes\mathbb{Z}_{\ell}\hookrightarrow H^{2}_{\mathrm{\acute{e}t}}(X_{\overline{\mathbb{F}}_{q}},\mathbb{Z}_{\ell}(1))
\end{equation*}
is an isomorphism because the cokernel $T_{\ell}\mathrm{Br}(X_{\overline{\mathbb{F}}_{q}})$ is torsion-free and the source and target have the same rank by definition of Shioda-supersingularity, and thus $H^{2}_{\mathrm{\acute{e}t}}(X_{\overline{\mathbb{F}}_{q}},\mathbb{Z}_{\ell}(n))\simeq(\mathrm{NS}(X_{\overline{\mathbb{F}}_{q}})\otimes\mathbb{Z}_{\ell})(n-1)$.
\end{proof}

\section{The $K$-theory of some surfaces over $\mathbb{F}_{q}$}\label{K-theory section}

\begin{prop}\label{K0}
Let $X$ be a K3 surface over a finite field $\mathbb{F}_{q}$. Then
\begin{equation*}
K_{0}(X)\simeq K_{0}(X_{\overline{\mathbb{F}}_{q}})^{\Gamma}\simeq\mathbb{Z}^{2+\rho(X)}
\end{equation*} 
where $\rho(X):=\mathrm{rank}\,\mathrm{NS(X)}$ is the Picard number of $X$.
\end{prop}
\begin{proof}
By \cite[\S12.2 Corollary 1.5]{Huy16}, the Grothendieck--Riemann-Roch theorem holds \emph{integrally} for K3 surfaces over algebraically closed fields. That is, the Chern character map defines a ring isomorphism
\begin{equation}\label{GRR}
\mathrm{ch}\,:\,K_{0}(X_{\overline{\mathbb{F}}_{q}})\xrightarrow{\sim}\bigoplus_{i=0}^{2}\mathrm{CH}^{i}(X_{\overline{\mathbb{F}}_{q}})\,.
\end{equation}
Clearly $\mathrm{CH}^{0}(X_{\overline{\mathbb{F}}_{q}})\simeq\mathbb{Z}$. It is well known that for a K3 surface over an algebraically closed field we have $\mathrm{CH}^{1}(X_{\overline{\mathbb{F}}_{q}})\simeq\mathrm{Pic}(X_{\overline{\mathbb{F}}_{q}})$ torsion-free and isomorphic to $\mathrm{NS}(X_{\overline{\mathbb{F}}_{q}})$. By \cite[\S12.3 Corollary 2.17]{Huy16} we have $\mathrm{CH}^{2}(X_{\overline{\mathbb{F}}_{q}})\simeq\mathbb{Z}$. In particular, 
\begin{equation*}
\,K_{0}(X_{\overline{\mathbb{F}}_{q}})\simeq\mathbb{Z}^{2+\rho(X_{\overline{\mathbb{F}}_{q}})}
\end{equation*}
is a free abelian group of rank $2+\rho(X_{\overline{\mathbb{F}}_{q}})$.

Taking Galois-invariants we see that
\begin{align*}
K_{0}(X_{\overline{\mathbb{F}}_{q}})^{\Gamma}
& \simeq\mathrm{CH}^{0}(X_{\overline{\mathbb{F}}_{q}}^{\Gamma})\oplus\mathrm{CH}^{1}(X_{\overline{\mathbb{F}}_{q}})^{\Gamma}\oplus\mathrm{CH}^{2}(X_{\overline{\mathbb{F}}_{q}})^{\Gamma}  \\
& \simeq \mathrm{CH}^{0}(X)\oplus\mathrm{Pic}(X_{\overline{\mathbb{F}}_{q}})^{\Gamma}\oplus\mathrm{CH}^{2}(X_{\overline{\mathbb{F}}_{q}})^{\Gamma} \\
& \simeq\mathrm{CH}^{0}(X)\oplus\mathrm{Pic}(X)\oplus\mathrm{CH}^{2}(X_{\overline{\mathbb{F}}_{q}})^{\Gamma}
\end{align*}
where the isomorphism $\mathrm{Pic}(X_{\overline{\mathbb{F}}_{q}})^{\Gamma}\simeq\mathrm{Pic}(X)$ is because the Hochschild-Serre spectral sequence gives an exact sequence
\begin{equation*}
0\rightarrow\mathrm{Pic}(X)\rightarrow\mathrm{Pic}(X_{\overline{\mathbb{F}}_{q}})^{\Gamma}\rightarrow\mathrm{Br}(\mathbb{F}_{q})
\end{equation*}
and the Brauer group of a finite field is trivial.

Now we claim that the natural map $\mathrm{CH}^{2}(X)\rightarrow\mathrm{CH}^{2}(X_{\overline{\mathbb{F}}_{q}})^{\Gamma}$ on zero-cycles is an isomorphism. Indeed, every geometrically irreducible projective variety over $\mathbb{F}_{q}$ has a zero-cycle of degree one (see \cite[Lemme 1]{Sou84} or \cite[Example 4.1.3]{Cla10}), so the composition (and hence the second map)
\begin{equation*}
\mathrm{deg}:\mathrm{CH}^{2}(X)\rightarrow\mathrm{CH}^{2}(X_{\overline{\mathbb{F}}_{q}})^{\Gamma}\rightarrow\mathbb{Z}
\end{equation*}
is surjective. But by \cite[\S12.3 Corollary 2.17]{Huy16} we have $\mathrm{CH}^{2}(X_{\overline{\mathbb{F}}_{q}})_{\mathrm{deg}=0}=0$, and by \cite[Proposition 8]{KS83} we have $\mathrm{CH}^{2}(X)_{\mathrm{deg}=0}=0$, so
\begin{equation*}
\mathrm{CH}^{2}(X)\xrightarrow{\sim}\mathrm{CH}^{2}(X_{\overline{\mathbb{F}}_{q}})^{\Gamma}\simeq\mathbb{Z}\,.
\end{equation*}
In particular, taking Galois-invariants of \eqref{GRR} gives an isomorphism
\begin{equation*}
K_{0}(X_{\overline{\mathbb{F}}_{q}})^{\Gamma}\xrightarrow{\sim}\bigoplus_{i=0}^{2}\mathrm{CH}^{i}(X)\,.
\end{equation*}

Finally we claim that the natural map $K_{0}(X)\rightarrow K_{0}(X_{\overline{\mathbb{F}}_{q}})^{\Gamma}$ is an isomorphism. Indeed, let $F^{2}K_{0}(X)\subset F^{1}K_{0}(X)\subset F^{0}K_{0}(X)=K_{0}(X)$ denote the coniveau filtration on $K_{0}(X)$, and similarly for $K_{0}(X_{\overline{\mathbb{F}}_{q}})$. Since $X$ is a smooth surface, we have isomorphisms
\begin{align*}
& F^{0}K_{0}(X)/F^{1}K_{0}(X)\xrightarrow{\sim}\mathrm{CH}^{0}(X) \\
& F^{1}K_{0}(X)/F^{2}K_{0}(X)\xrightarrow{\sim}\mathrm{CH}^{1}(X) \\
& F^{2}K_{0}(X)\xrightarrow{\sim}\mathrm{CH}^{2}(X)\,.
\end{align*}
and similarly for $X_{\overline{\mathbb{F}}_{q}}$ \cite[Example 15.3.6]{Ful84}. By our previous discussion, we see then that the map $K_{0}(X)\rightarrow K_{0}(X_{\overline{\mathbb{F}}_{q}})^{\Gamma}$ induces an isomorphism on associated graded groups, and is therefore itself an isomorphism. 
\end{proof}
 
We also can describe the higher $K$-groups for more general surfaces, if we assume Parshin's conjecture (see Appendix \ref{Parshin section}):

\begin{thm}\label{AHSS}
Let $X$ be a smooth projective geometrically irreducible surface over $\mathbb{F}_{q}$ with $H^{1}_{\mathrm{\acute{e}t}}(X_{\overline{\mathbb{F}}_{q}},\mathbb{Z}_{\ell})=H^{3}_{\mathrm{\acute{e}t}}(X_{\overline{\mathbb{F}}_{q}},\mathbb{Z}_{\ell})=0$ and $H^{2}_{\mathrm{\acute{e}t}}(X_{\overline{\mathbb{F}}_{q}},\mathbb{Z}_{\ell})$ a free $\mathbb{Z}_{\ell}$-module, for $\ell\neq p$. If Parshin's conjecture holds for $X$, then the motivic Atiyah-Hirzebruch spectral sequence
\begin{equation*}
E_{2}^{i,j}=H^{i-j}_{\mathcal{M}}(X,\mathbb{Z}(-j))\Rightarrow K_{-i-j}(X)
\end{equation*}
degenerates at the $E_{2}$-page. Moreover, the higher $K$-groups of $X$ are as follows:
\[ K_{n}(X)\simeq \begin{cases} 
      \left(\mathbb{Z}/(q-1)\mathbb{Z}\right)^{\oplus 2}\oplus\displaystyle\varinjlim_{r}H^{0}_{\mathrm{Zar}}(X,W_{r}\Omega_{X,\log}^{2})\oplus\bigoplus_{\ell\neq p}{H^{2}_{\mathrm{\acute{e}t}}(X_{\overline{\mathbb{F}}_{q}},\mathbb{Z}_{\ell}(2))}_{\Gamma} & \text{ if }n=1 \\
      0 & \text{ if } n=2m, m\geq 1 \\
      \left(\mathbb{Z}/(q^{m}-1)\mathbb{Z}\right)^{\oplus 2}\oplus\displaystyle\bigoplus_{\ell\neq p}{H^{2}_{\mathrm{\acute{e}t}}(X_{\overline{\mathbb{F}}_{q}},\mathbb{Z}_{\ell}(m+1))}_{\Gamma} & \text{ if }n=2m-1, m\geq 2\,. \\
   \end{cases}
\]
\end{thm} 
\begin{proof}
We have computed the $E_{2}$-page as a special case of Theorem \ref{main} (see Figure \ref{fig:M1}). We see that the only differential to investigate is 
\begin{equation*}
d_{2}:E^{0,-1}_{2}=\mathcal{O}^{\ast}(X)=\mathbb{Z}/(q-1)\mathbb{Z}\rightarrow E^{2,-2}_{2}=\mathrm{CH}_{0}(X)\,.
\end{equation*}
We claim it is the zero map. Since the higher $K$-groups $K_{n+1}(X)$, $n\geq 0$, are torsion (because $X$ satisfies Parshin's conjecture), the universal coefficient exact sequence 
\begin{equation*}
0\rightarrow K_{n+1}(X)\otimes\mathbb{Q}_{\ell}/\mathbb{Z}_{\ell}\rightarrow K_{n+1}(X;\mathbb{Q}_{\ell}/\mathbb{Z}_{\ell})\rightarrow K_{n}(X)_{\ell\mathrm{-tors}}\rightarrow 0
\end{equation*}  
shows that $K_{n}(X)=K_{n}(X)_{\mathrm{tors}}\simeq\bigoplus_{\ell}K_{n+1}(X;\mathbb{Q}_{\ell}/\mathbb{Z}_{\ell})$, where $\ell$ ranges over all primes. For each prime $\ell$, the maps $K_{n+1}(X;\mathbb{Q}_{\ell}/\mathbb{Z}_{\ell})\rightarrow K_{n}(X)_{\ell\mathrm{-tors}}$ (and the analogous maps from motivic cohomology with $\mathbb{Q}_{\ell}/\mathbb{Z}_{\ell}$-coefficients) give a map of spectral sequences from the Atiyah-Hirzebruch spectral sequence with $\mathbb{Q}_{\ell}/\mathbb{Z}_{\ell}$-coefficients. In particular, the edge maps fit into a commutative diagram 
\begin{equation*}
\begin{tikzpicture}[descr/.style={fill=white,inner sep=1.5pt}]
        \matrix (m) [
            matrix of math nodes,
            row sep=2.5em,
            column sep=2em,
            text height=1.5ex, text depth=0.25ex
        ]
        { \displaystyle\bigoplus_{\ell\neq p}K_{2}(X;\mathbb{Q}_{\ell}/\mathbb{Z}_{\ell}) & \displaystyle\bigoplus_{\ell\neq p}H^{0}_{\mathcal{M}}(X,\mathbb{Q}_{\ell}/\mathbb{Z}_{\ell}(1)) \\
        K_{1}(X) & H^{1}_{\mathcal{M}}(X,\mathbb{Z}(1)) \\};

        \path[overlay,->, font=\scriptsize] 
        (m-1-1) edge node[above]{edge}(m-1-2)
        (m-2-1) edge node[above]{edge}(m-2-2)
        (m-1-2) edge node [right] {$\wr$} (m-2-2);
        
        \path[overlay, right hook->, font=\scriptsize]
        (m-1-1) edge (m-2-1);
                                                
\end{tikzpicture}
\end{equation*}
The vertical maps are the inclusion of the $\ell\neq p$-primary summands; the right-hand vertical map is an isomorphism because $H^{1}_{\mathcal{M}}(X,\mathbb{Z}(1))= H^{1}_{\mathcal{M}}(X,\mathbb{Z}(1))_{\mathrm{tors}}=\mathcal{O}^{\ast}(X)$. The top edge map is a split surjection. In fact, more generally, the edge maps $K_{2n}(X;\mathbb{Q}_{\ell}/\mathbb{Z}_{\ell})\rightarrow H^{0}_{\mathrm{\acute{e}t}}(X,\mathbb{Q}_{\ell}/\mathbb{Z}_{\ell}(n))$ coming from the Atiyah-Hirzebruch spectral sequence with $\mathbb{Q}_{\ell}/\mathbb{Z}_{\ell}$-coefficients are split surjections for each $n\geq 1$ and hence $\mathbb{Z}/(q^{n}-1)\mathbb{Z}\simeq H^{1}_{\mathcal{M}}(X,\mathbb{Z}(n))\simeq\bigoplus_{\ell\neq p}H^{0}_{\mathrm{\acute{e}t}}(X,\mathbb{Q}_{\ell}/\mathbb{Z}_{\ell}(n))$ is a direct summand of the odd $K$-groups $K_{2n-1}(X)$ by \cite[Corollary 9.6]{Kah97} (it is the so-called Bott summand). 

Now, the Atiyah-Hirzebruch spectral sequence determines the $K$-groups up to an extension problem; the graded quotients of the induced filtration on $K_{n}(X)$ are the integral motivic cohomology groups $H_{\mathcal{M}}^{2i-n}(X,\mathbb{Z}(i))$. We claim that these extensions split for $n\geq 1$. For the even $K$-groups $K_{2n}(X)$, $n\geq 1$, there is nothing to prove since the graded pieces are trivial. For the odd $K$-groups $K_{2n-1}(X)$, we have seen that $\mathbb{Z}/(q^{n}-1)\mathbb{Z}\simeq H^{1}_{\mathcal{M}}(X,\mathbb{Z}(n))$ is a direct summand (the Bott summand). To see that the other copy of $\mathbb{Z}/(q^{n}-1)\mathbb{Z}$ appearing as a graded quotient of $K_{2n-1}(X)$ is also a direct summand, first suppose that $X$ has an $\mathbb{F}_{q}$-rational point $i:\mathrm{Spec}\,\mathbb{F}_{q}\rightarrow X$. Then  $i_{\ast}K_{2n-1}(\mathbb{F}_{q})\simeq\mathbb{Z}/(q^{n}-1)\mathbb{Z}$ is a summand of $K_{2n-1}(X)$, proving the result. If $X$ does not have an $\mathbb{F}_{q}$-rational point, then choose an $F$-rational point for a finite extension $F/\mathbb{F}_{q}$, and use that $K_{2n-1}(F)^{G}\simeq K_{2n-1}(\mathbb{F}_{q})$ where $G=\mathrm{Gal}(F/\mathbb{F}_{q})$. 
\end{proof}

\begin{rem}\label{Coombes rational}
The Brown-Gersten-Quillen spectral sequence 
\begin{equation*}
E_{2}^{i,j}=H^{i}(X,\mathcal{K}_{-j})\Rightarrow K_{-i-j}(X)
\end{equation*}
degenerates at the $E_{2}$-page for rational surfaces, and the filtration on the abutment splits to give a direct sum decomposition
\begin{equation*}
K_{n}(X)\simeq H^{0}(X,\mathcal{K}_{n})\oplus H^{1}(X,\mathcal{K}_{n+1})\oplus  H^{2}(X,\mathcal{K}_{n+2})
\end{equation*}
for all $n\geq 0$. The $K$-cohomology groups, up to $p$-primary torsion, of a rational surface over $\mathbb{F}_{q}$ are computed in \cite{Coo87}:
\begin{equation*}
H^{0}(X,\mathcal{K}_{n})\simeq H^{2}(X,\mathcal{K}_{n+2})\simeq K_{n}(\mathbb{F}_{q})
\end{equation*} 
\begin{equation*}
H^{1}(X,\mathcal{K}_{n+1})\simeq(\mathrm{Pic}(X_{\overline{\mathbb{F}}_{q}})\otimes K_{n}(\overline{\mathbb{F}}_{q}))^{\Gamma}\,. 
\end{equation*}
One finds then that
\[ K_{n}(X)\simeq \begin{cases} 
      \mathbb{Z}^{2}\oplus\mathrm{Pic}(X) & \text{ if } n=0 \\
      0 & \text{ if }n=2m, m\geq 1 \\
      \left(\mathbb{Z}/(q^{m}-1)\mathbb{Z}\right)^{\oplus 2}\oplus\left(\mathrm{Pic}(X_{\overline{\mathbb{F}}_{q}})\otimes K_{2m-1}(\overline{\mathbb{F}}_{q})\right)^{\Gamma} & \text{ if }n=2m-1, m\geq 1
   \end{cases}
\]
up to $p$-primary torsion. Note that $K_{2m-1}(\overline{\mathbb{F}}_{q})\simeq\bigoplus_{\ell\neq p}\mathbb{Q}_{\ell}/\mathbb{Z}_{\ell}(m)$ by Quillen's computation of $K$-theory for finite fields, and $\mathrm{Pic}(X_{\overline{\mathbb{F}}_{q}})\otimes \mathbb{Q}_{\ell}/\mathbb{Z}_{\ell}\simeq H^{2}_{\mathrm{\acute{e}t}}(X_{\overline{\mathbb{F}}_{q}},\mathbb{Q}_{\ell}/\mathbb{Z}_{\ell}(1))$ because $\mathrm{Br}(X_{\overline{\mathbb{F}}_{q}})=0$ for a rational surface. The snake lemma on the diagram
\begin{equation*}
\begin{tikzpicture}[descr/.style={fill=white,inner sep=1.5pt}]
        \matrix (m) [
            matrix of math nodes,
            row sep=2.5em,
            column sep=1em,
            text height=1.5ex, text depth=0.25ex
        ]
        { 0 & H^{2}_{\mathrm{\acute{e}t}}(X_{\overline{\mathbb{F}}_{q}},\mathbb{Z}_{\ell}(m+1)) & H^{2}_{\mathrm{\acute{e}t}}(X_{\overline{\mathbb{F}}_{q}},\mathbb{Q}_{\ell}(m+1)) & H^{2}_{\mathrm{\acute{e}t}}(X_{\overline{\mathbb{F}}_{q}},\mathbb{Q}_{\ell}/\mathbb{Z}_{\ell}(m+1)) & 0 \\
       0 & H^{2}_{\mathrm{\acute{e}t}}(X_{\overline{\mathbb{F}}_{q}},\mathbb{Z}_{\ell}(m+1)) & H^{2}_{\mathrm{\acute{e}t}}(X_{\overline{\mathbb{F}}_{q}},\mathbb{Q}_{\ell}(m+1)) & H^{2}_{\mathrm{\acute{e}t}}(X_{\overline{\mathbb{F}}_{q}},\mathbb{Q}_{\ell}/\mathbb{Z}_{\ell}(m+1)) & 0 \\};

        \path[overlay,->, font=\scriptsize] 
        (m-1-1) edge (m-1-2)
        (m-1-2) edge (m-1-3)
        (m-1-3) edge (m-1-4)
        (m-1-4) edge (m-1-5)
        (m-2-1) edge (m-2-2)
        (m-2-2) edge (m-2-3)
        (m-2-3) edge (m-2-4)
        (m-2-4) edge (m-2-5)
        (m-1-2) edge node[right]{$F-1$} (m-2-2)
        (m-1-3) edge node[right]{$F-1$} (m-2-3)
        (m-1-4) edge node[right]{$F-1$} (m-2-4)
        ;
                                                
\end{tikzpicture}
\end{equation*}
shows that $\left(\mathrm{Pic}(X_{\overline{\mathbb{F}}_{q}})\otimes K_{2m-1}(\overline{\mathbb{F}}_{q})\right)^{\Gamma}\simeq\bigoplus_{\ell\neq p}{H^{2}_{\mathrm{\acute{e}t}}(X_{\overline{\mathbb{F}}_{q}},\mathbb{Z}_{\ell}(m+1))}_{\Gamma}$. In particular, the calculation in Theorem \ref{AHSS} in the special case of rational surfaces recovers the above result of Coombes and extends them to include $p$-primary torsion.
\end{rem}

\begin{example}
Suppose that $\mathrm{char}(\mathbb{F}_{q})=p$ is odd. Consider the Fermat surface 
\begin{equation*}
S_{d}:=\{X^d+Y^d+Z^d+W^d=0\}\subset\mathbb{P}_{\mathbb{F}_{q}}^{3}
\end{equation*}
for $d\geq 1$ with $d\not\equiv 0\mod p$. Suppose that $p^{\nu}+1\equiv 0\mod d$ for some power $\nu$. Then  
$S_{d}$ is unirational and hence supersingular \cite{Shi74}. In particular, Parshin's conjecture holds for $S_{d}$ and $H^{2}_{\mathrm{\acute{e}t}}((S_{d})_{\overline{\mathbb{F}}_{q}},\mathbb{Z}_{\ell}(n))\simeq(\mathrm{NS}((S_{d})_{\overline{\mathbb{F}}_{q}})\otimes\mathbb{Z}_{\ell})(n-1)$ for all $\ell\neq p$. Note that $\varinjlim_{r}H^{0}_{\mathrm{Zar}}(S_{d},W_{r}\Omega_{S_{d},\log}^{2})=0$ if $d\leq 3$ by Remark \ref{no p-torsion}. In fact, by \cite{AM77} we have $H^{2}(S_{d},W\mathcal{O}_{S_{d}})\simeq D(\widehat{\mathrm{Br}}_{S_{d}})=D(\widehat{\mathbb{G}}_{a}^{\oplus h^{0,2}})=\mathbb{F}_{q}\llbracket x\rrbracket^{\oplus h^{0,2}}$, where $h^{0,2}=\dim_{\mathbb{F}_{q}}H^{2}(S_{d},\mathcal{O}_{S_{d}})={d-1\choose 3}$, with $F=0, Vx^{m}=x^{m+1}$, and hence $\left[\varprojlim_{\sigma}H^{2}_{\mathrm{\acute{e}t}}(X,W\mathcal{O}_{X})\right]_{\sigma}=0$. Hence we conclude from the proof of Theorem \ref{main} that $\varinjlim_{r}H^{0}_{\mathrm{Zar}}(X,W_{r}\Omega_{X,\log}^{2})=0$ for all $d$. Thus Theorem \ref{AHSS} gives
\[ K_{n}(S_{d})\simeq \begin{cases} 
           0 & \text{ if } n=2m, m\geq 1 \\
      \left(\mathbb{Z}/(q^{m}-1)\mathbb{Z}\right)^{\oplus 2}\oplus\displaystyle\bigoplus_{\ell\neq p}{(\mathrm{NS}((S_{d})_{\overline{\mathbb{F}}_{q}})\otimes\mathbb{Z}_{\ell})(m)}_{\Gamma} & \text{ if }n=2m-1, m\geq 1\,. \\
   \end{cases}
\]
\end{example}

\section{Motivic cohomology and $K$-theory of Enriques surfaces}

It is shown in \cite[Proposition 3.1]{Coo92} that the Chow $\mathbb{Z}[1/2]$-motive of an Enriques surface $X$ over an arbitrary field $k$ with an elliptic pencil $E\rightarrow\mathbb{P}_{k}^{1}$ is isomorphic to the Chow $\mathbb{Z}[1/2]$-motive of the rational surface $J$ given by the associated Jacobian fibration $J\rightarrow\mathbb{P}_{k}^{1}$. Now consider an Enriques surface $X$ over $\mathbb{F}_{q}$ with $\mathrm{char}(\mathbb{F}_{q})=p>2$. In \cite[\S2]{Coo92} it is shown that there exists a finite extension $F/\mathbb{F}_{q}$ such that the base change $X_{F}$ admits an elliptic pencil. Moreover, it is shown in \cite[Theorem 3.2]{Coo92} that the Chow $\mathbb{Z}[1/2p]$-motive of $X$ is completely determined by the Chow $\mathbb{Z}[1/2p]$-motive of $X_{F}$. In particular, the Chow $\mathbb{Z}[1/2p]$-motive of $X$, and hence the $K$-cohomology of $X$ up to $2$- and $p$-torsion, agrees with that of a rational surface. As we mentioned in Remark \ref{Coombes rational}, the $K$-cohomology and $K$-theory of rational surfaces over finite fields was computed in \cite{Coo87}, up to $p$-torsion. In this way, Coombes computes the $K$-theory of an Enriques surface over $\mathbb{F}_{q}$ up to $2$- and $p$-torsion. For completeness, we include here a description of the motivic cohomology of Enriques surfaces, and confirm Coombes' suspicion that the $p$-torsion vanishes \cite[Remark 3.5]{Coo92}. We also handle the $2$-torsion.
\begin{thm}
Let $X$ be an Enriques surface over $\mathbb{F}_{q}$, with $\mathrm{char}(\mathbb{F}_{q})=p>2$. Then the motivic cohomology groups $H^{i}_{\mathcal{M}}(X,\mathbb{Z}(n))$ of $X$ are as in Table \ref{fig:table2}. Moreover, the motivic Atiyah-Hirzebruch spectral sequence 
\begin{equation*}
E_{2}^{i,j}=H^{i-j}_{\mathcal{M}}(X,\mathbb{Z}(-j))\Rightarrow K_{-i-j}(X)
\end{equation*}
of $X$ degenerates at the $E_{2}$-page, and the $K$-groups $K_{n}(X)$ of $X$ are as follows:
\[ K_{n}(X)\simeq \begin{cases} 
      \mathbb{Z}^{2}\oplus\mathrm{Pic}(X) & \text{ if } n=0 \\
      \mathbb{Z}/2\mathbb{Z} & \text{ if }n=2m, m\geq 1 \\
      \left(\mathbb{Z}/(q^{m}-1)\mathbb{Z}\right)^{\oplus 2}\oplus\displaystyle\bigoplus_{\ell\neq p}\left(\mathrm{Pic}(X_{\overline{\mathbb{F}}_{q}})\otimes\mathbb{Z}_{\ell}(m)\right)_{\Gamma} & \text{ if }n=2m-1, m\geq 1\,.
   \end{cases}
\]
\end{thm}
\begin{proof}
By the above discussion, Parshin's conjecture holds for $X$, so the motivic cohomology is torsion away from the $(2n,n)$ Chow diagonal. Since we are in characteristic $\neq 2$, the Enriques surface $X$ is ``classical'' and the \'{e}tale cohomology for $\ell\neq p$ of $X_{\overline{\mathbb{F}}_{q}}$ is as follows:
\begin{equation*}
H^{i}_{\mathrm{\acute{e}t}}(X_{\overline{\mathbb{F}}_{q}},\mathbb{Z}_{\ell})\simeq\begin{cases} 
      \mathbb{Z}_{\ell} & \text{ if }i=0 \\
      0 & \text{ if } i=1\\
      \mathrm{Pic}(X_{\overline{\mathbb{F}}_{q}})\otimes\mathbb{Z}_{\ell}(-1) & \text{ if }i=2 \\
      0 & \text{ if }i=3\text{ and }\ell\neq 2 \\
      \mathbb{Z}/2\mathbb{Z}(-1) & \text{ if }i=3\text{ and }\ell=2 \\
      \mathbb{Z}_{\ell}(-2) & \text{ if }i=4\,.
   \end{cases}
\end{equation*}
The Picard group $\mathrm{Pic}(X_{\overline{\mathbb{F}}_{q}})$ is $\mathbb{Z}^{10}\oplus\mathbb{Z}/2\mathbb{Z}$. The $2$-torsion part of $\mathrm{Pic}(X_{\overline{\mathbb{F}}_{q}})$ is already defined over $\mathbb{F}_{q}$ because it is generated by the canonical invertible sheaf, hence $\Gamma$ acts trivially on the $\mathbb{Z}/2\mathbb{Z}$ summand.

The computation now proceeds along the lines of the proof of Theorem \ref{main}; using Lemma \ref{torsion in motivic cohomology lemma} and Lemma \ref{GL} we are reduced to computing $H^{i}_{\mathrm{\acute{e}t}}(X,\mathbb{Z}_{\ell}(n))$ in the range $0\leq i\leq 6$ for primes $\ell\neq p$ to find the $\ell\neq p$-primary summands. These groups are immediate from the Hochschild-Serre exact sequence \eqref{HS sequence} and the above description of $H^{i}_{\mathrm{\acute{e}t}}(X_{\overline{\mathbb{F}}_{q}},\mathbb{Z}_{\ell})$. The same argument as in the proof of Theorem \ref{main} shows that the only possible $p$-torsion appears in $H^{3}_{\mathcal{M}}(X,\mathbb{Z}(2))$, where it is given by $\varinjlim_{r}H^{0}_{\mathrm{Zar}}(X,W_{r}\Omega^{2}_{X,\log})$. But $H^{0}(X,\Omega^{2}_{X})=0$ so $H^{0}_{\mathrm{Zar}}(X,W_{r}\Omega^{2}_{X,\log})=0$ by Remark \ref{no p-torsion}. We see then that the $E_{2}$-page of the motivic Atiyah-Hirzebruch spectral sequence for $X$ is as in Figure \ref{fig:M2}. The same argument as in the proof of Theorem \ref{AHSS} shows that the only possible non-trivial differentials are in fact the zero maps. (Alternatively, one can argue using Coombes' result that the kernels up to $2$-torsion are of the form $\mathbb{Z}/(q^{n}-1)\mathbb{Z}$, so cannot be the non-trivial map to $\mathbb{Z}/2\mathbb{Z}$. The differential labelled $d_{2}$ is clearly the zero map because $\mathrm{CH}_{0}(X)=\mathbb{Z}$ and $\mathcal{O}^{\ast}(X)=\mathbb{Z}/(q-1)\mathbb{Z}$.). This proves degeneration at $E_{2}$ and gives the $K$-groups. Note that $K_{2m-1}(\overline{\mathbb{F}}_{q})\simeq\bigoplus_{\ell\neq p}\mathbb{Q}_{\ell}/\mathbb{Z}_{\ell}(m)$ by Quillen's computation of $K$-theory for finite fields, and $\bigoplus_{\ell\neq 2,p}\mathrm{Pic}(X_{\overline{\mathbb{F}}_{q}})\otimes\mathbb{Q}_{\ell}/\mathbb{Z}_{\ell}\simeq\bigoplus_{\ell\neq 2,p}H^{2}_{\mathrm{\acute{e}t}}(X_{\overline{\mathbb{F}}_{q}},\mathbb{Q}_{\ell}/\mathbb{Z}_{\ell}(1))$ for $\ell\neq 2,p$ because $\mathrm{Br}(X_{\overline{\mathbb{F}}_{q}})=\mathbb{Z}/2\mathbb{Z}$, so our result is consistent with \cite{Coo92} the same argument as in Remark \ref{Coombes rational}.
\end{proof}

\begin{rem}
When $\mathrm{char}(\mathbb{F}_{q})=2$ there are non-classical Enriques surfaces (\emph{``singular''} and \emph{``supersingular''} in the terminology of \cite{BM76}). It would be interesting to compute the $K$-theory of these non-classical Enriques surfaces.
\end{rem}

\appendix

\section{Parshin's conjecture for some surfaces}\label{Parshin section} 

In this appendix we shall make the small observation that a smooth projective and geometrically irreducible surface $X$ over $\mathbb{F}_{q}$ satisfies Parshin's conjecture if $X_{\overline{\mathbb{F}}_{q}}$ (equivalently $X_{\mathbb{F}}$ for some finite extension $\mathbb{F}/\mathbb{F}_{q}$) admits a rational decomposition of the diagonal (in the sense of Bloch-Srinivas \cite{BS83}). The proof is by combining work of Pedrini \cite{Ped00} with Quillen's computation of the $K$-theory of finite fields \cite{Qui72}. This class contains the class of unirational surfaces; indeed the pushforward of an integral decomposition of the diagonal for $\mathbb{P}^{2}_{\overline{\mathbb{F}}_{q}}$ along a dominant rational map $\mathbb{P}^{2}_{\overline{\mathbb{F}}_{q}}\dashrightarrow X$ induces a rational decomposition of the diagonal on $X_{\overline{\mathbb{F}}_{q}}$. Finally we shall explain how known results imply that Parshin's conjecture holds for K3 surfaces $X$ with geometric Picard number $\rho(X_{\overline{\mathbb{F}}_{q}})=20$.

\subsection{Parshin's conjecture and rational decomposition of the diagonal}

Recall that a smooth proper variety $X$ of dimension $d$ over a field $k$ admits a rational decomposition of the diagonal if there exists an integer $N\geq 1$ such that $N$ times the class of the diagonal $\Delta_{X}\in\mathrm{CH}^{d}(X\times X)$ decomposes as
\begin{equation*}
N\Delta_{X}=P\times X+Z\in\mathrm{CH}^{d}(X\times X)
\end{equation*}
where $P$ is a $0$-cycle of degree $N$ and $Z$ is a cycle with support in $X\times V$ for $V\subsetneq X$ some closed subvariety. 

Let $A_{0}(X):=\ker(\mathrm{deg}:\mathrm{CH}_{0}(X)\rightarrow\mathbb{Z})$ denote the group of $0$-cycles of degree zero on $X$. Recall that if $X$ admits a rational decomposition of the diagonal, then $A_{0}(X)$ is torsion (even $N$-torsion, for $N$ the integer appearing in a rational decomposition of the diagonal). Indeed, the correspondence $[N\Delta_{X}]_{\ast}$ acts on $\mathrm{CH}_{0}(X)$ as multiplication-by-$N$. But $[P\times X]_{\ast}$ acts as $\mathrm{deg}(-)P$ and $[Z]_{\ast}$ acts as the zero map (every $0$-cycle on $X$ is rationally equivalent to a $0$-cycle supported away from $V$, by the moving lemma), so we see that the multiplication-by-$N$ map is equal to the zero map on $A_{0}(X)$.

In fact, let $N\geq 1$ be an integer. If $X$ is also geometrically irreducible, then the following are equivalent (\cite{BS83}, but see also \cite[Theorem 5.1.3]{AB17}):
\begin{enumerate}
\item The degree map $\mathrm{deg}:\mathrm{CH}_{0}(X_{K})\rightarrow\mathbb{Z}$ is surjective and has $N$-torsion kernel for every field extension $K/k$.
\item The variety $X$ has a $0$-cycle of degree $1$ and the kernel of $\mathrm{deg}:\mathrm{CH}_{0}(X_{k(X)})\rightarrow\mathbb{Z}$ is $N$-torsion.
\item The variety $X$ has a rational decomposition of the diagonal of the form 
\begin{equation*}
N\Delta_{X}=N(P\times X)+Z
\end{equation*}
for a $0$-cycle $P$ of degree $1$ on $X$.
\end{enumerate}

\begin{thm}\label{Pedrini's theorem}
Let $X$ be a smooth irreducible projective surface over an algebraically closed field $k$, and suppose that $X$ admits a rational decomposition of the diagonal. Then
\begin{equation*}
K_{n}(X)_{\mathbb{Q}}\simeq (K_{0}(X)\otimes K_{n}(k))_{\mathbb{Q}}
\end{equation*} 
for all $n\geq 0$.
\end{thm}
\begin{proof}
This result is due to Pedrini \cite[\S2]{Ped00} but is formulated slightly differently. We shall give a brief outline of his proof. Let
\begin{equation*}
\tau_{m,n}:\mathrm{CH}^{m}(X)\otimes K_{n}(k)\rightarrow H^{m}(X,\mathcal{K}_{m+n})
\end{equation*}
denote the natural pairing of $K$-cohomology groups induced by Milnor's pairing $K_{m}\otimes K_{n}\rightarrow K_{m+n}$ in $K$-theory, and the Bloch-Quillen isomorphism $H^{m}(X,\mathcal{K}_{m})\simeq\mathrm{CH}^{m}(X)$. By \cite[Theorem 3]{BS83}, the maps $\tau_{0,n}\otimes\mathbb{Q}$ and $\tau_{1,n}\otimes\mathbb{Q}$ are surjective for all $n\geq 0$. The maps $\tau_{0,n}$ are injective for all $n$, hence the $\tau_{0,n}\otimes\mathbb{Q}$ are isomorphisms. By \cite[Prop 2.3]{Ped00}, the kernel of $\tau_{1,n}$ is contained within $\mathrm{Pic}^{0}(X)\otimes K_{n}(k)$. But $\mathrm{Pic}^{0}(X)\otimes\mathbb{Q}$ is trivial -- indeed, the Albanese map 
\begin{equation*}
\mathrm{alb}_{X}:A_{0}(X)\rightarrow\mathrm{Alb}_{X}(k)
\end{equation*}
on zero-cycles of degree zero is surjective (because $k$ is algebraically closed). But we have seen that $A_{0}(X)_{\mathbb{Q}}=0$ under our hypotheses. Hence the maps $\tau_{1,n}\otimes\mathbb{Q}$ are isomorphisms for all $n\geq 0$ too. 

It follows from the Gersten resolution of $\mathcal{K}_{n}$ (see \cite[Lemma 2.5]{Ped00}) that there is a split short exact sequence 
\begin{equation*}
0\rightarrow(A_{0}(X)\otimes K_{n}(k))/\ker\tau_{2,n}\rightarrow H^{2}(X,\mathcal{K}_{n+2})\rightarrow K_{n}(k)\rightarrow 0
\end{equation*}
for each $n\geq 0$. In particular, since $A_{0}(X)_{\mathbb{Q}}=0$, we see that the maps $\tau_{2,n}\otimes\mathbb{Q}$ are isomorphisms as well.

Now consider the Brown-Gersten-Quillen spectral sequence for $K$-cohomology:
\begin{equation*}
E_{2}^{i,j}=H^{i}(X,\mathcal{K}_{-j})\Rightarrow K_{-i-j}(X)\,.
\end{equation*}
This is a fourth-quadrant spectral sequence, and since $X$ is a surface, the $E_{2}$-page looks as follows:

\begin{center}
\begin{tikzpicture}
  \matrix (m) [matrix of math nodes,
    nodes in empty cells,nodes={minimum width=5ex,
    minimum height=5ex,outer sep=-5pt},
    column sep=1ex,row sep=1ex]{
             \ddots & \vdots & \vdots & \vdots & \vdots & \vdots & \reflectbox{$\ddots$} \\
               \cdots & 0 & 0 & 0 & 0 & 0 & \cdots \\
          \cdots & 0 \ \ & \mathbb{Z}(X) & 0 & 0 & 0 & \cdots \\   
         \cdots &  0 \ \ & \mathcal{O}^{\ast}(X) & \mathrm{Pic}(X) & 0 & 0 & \cdots \\
          \cdots & 0 \ \ & \mathcal{K}_{2}(X) & H^{1}(X,\mathcal{K}_{2}) &  H^{2}(X,\mathcal{K}_{2})  & 0 & \cdots \\
          \cdots & 0 \ \ & \mathcal{K}_{3}(X) & H^{1}(X,\mathcal{K}_{3}) &  H^{2}(X,\mathcal{K}_{3})  & 0 & \cdots \\
          \cdots & 0 \ \ & \mathcal{K}_{4}(X) & H^{1}(X,\mathcal{K}_{4}) &  H^{2}(X,\mathcal{K}_{4})  & 0 & \cdots \\
    \quad\strut  \reflectbox{$\ddots$} & \vdots \ \ & \vdots & \vdots & \vdots & \vdots & \ddots \strut \\};
    \draw[-stealth] (m-4-3.south east) -- (m-5-5.north west);
    \draw[-stealth] (m-5-3.south east) -- (m-6-5.north west);
    \draw[-stealth] (m-6-3.south east) -- (m-7-5.north west);
\end{tikzpicture}
\end{center}
The differentials $d_{2}:\mathcal{K}_{n}(X)\rightarrow H^{2}(X,\mathcal{K}_{n+1})$ are the only possible non-trivial differentials. By \cite[Proposition 2.6]{Ped00}, the surjectivity of the maps $\tau_{0,n}\otimes\mathbb{Q}$ force the $d_{2}$ differentials to be torsion (i.e. $d_{2}\otimes\mathbb{Q}=0$), and hence the spectral sequence degenerates at $E_{2}$ up to torsion. In particular, 
\begin{equation*}
K_{n}(X)_{\mathbb{Q}}\simeq H^{0}(X,\mathcal{K}_{n})_{\mathbb{Q}}\oplus H^{1}(X,\mathcal{K}_{n+1})_{\mathbb{Q}}\oplus H^{2}(X,\mathcal{K}_{n+2})_{\mathbb{Q}}\,.
\end{equation*}
Applying the isomorphisms $\tau_{m,n}\otimes\mathbb{Q}$ (for $m=0,1,2$) and Grothendieck-Riemann-Roch, we get 
\begin{align*}
K_{n}(X)_{\mathbb{Q}}
& \simeq K_{n}(k)_{\mathbb{Q}}\oplus (\mathrm{CH}^{1}(X)\otimes K_{n}(k))_{\mathbb{Q}}\oplus (\mathrm{CH}^{2}(X)\otimes K_{n}(k))_{\mathbb{Q}} \\
& \simeq (K_{0}(X)\otimes K_{n}(k))_{\mathbb{Q}}
\end{align*}
as desired.
\end{proof}

\begin{cor}\label{Parshin}
Let $X$ be a smooth geometrically irreducible projective surface over a finite field $\mathbb{F}_{q}$, and suppose that $X_{\overline{\mathbb{F}}_{q}}$ admits a rational decomposition of the diagonal. Then
\begin{equation*}
K_{n}(X)_{\mathbb{Q}}=0
\end{equation*}
for all $n\geq 1$. That is, $X$ satisfies Parshin's conjecture.
\end{cor}
\begin{proof}
By Theorem \ref{Pedrini's theorem}, we have
\begin{equation*}
K_{n}(X_{\overline{\mathbb{F}}_{q}})_{\mathbb{Q}}\simeq (K_{0}(X_{\overline{\mathbb{F}}_{q}})\otimes K_{n}(\overline{\mathbb{F}}_{q}))_{\mathbb{Q}}
\end{equation*}
for all $n\geq 0$. By Quillen's computation of the $K$-theory of finite fields \cite{Qui72}, we have that $K_{n}(\overline{\mathbb{F}}_{q})=\bigcup_{s\geq 1} K_{n}(\mathbb{F}_{q^{s}})$ is torsion for all $n\geq 1$. In particular, 
\begin{equation*}
K_{n}(X_{\overline{\mathbb{F}}_{q}})_{\mathbb{Q}}=0
\end{equation*}
for all $n\geq 1$. By \'{e}tale descent for rational $K$-theory \cite[Theorem 2.15]{Tho85} we conclude that 
\begin{equation*}
K_{n}(X)_{\mathbb{Q}}=K_{n}(X_{\overline{\mathbb{F}}_{q}})_{\mathbb{Q}}^{\Gamma}=0
\end{equation*}
where $\Gamma:=\mathrm{Gal}(\overline{\mathbb{F}}_{q}/\mathbb{F}_{q})$.
\end{proof}

\subsection{Parshin's conjecture for some K3 surfaces}
\
\\
Let $X$ be a K3 surface over $\mathbb{F}_{q}$. Then the geometric Picard number $\rho(X_{\overline{\mathbb{F}}_{q}}):=\mathrm{rank}\,\mathrm{NS}(X_{\overline{\mathbb{F}}_{q}})$ lies between $1$ and $22$. In fact, Swinnerton-Dyer observed that the Tate conjecture forces $\rho(X_{\overline{\mathbb{F}}_{q}})$ to be an even number \cite[p. 544]{Art74}. The extreme case $\rho(X_{\overline{\mathbb{F}}_{q}})=22$ is the (Shioda-)supersingular case. It is a conjecture of Artin-Rudakov-Shafarevich-Shioda that supersingular K3 surface are unirational, and in particular satisfy Parshin's conjecture. This is known in complete generality only in characteristic $2$ \cite{RS78}. It is also known for Kummer surfaces \cite{Shi77}. After supersingular K3 surfaces, the next most accessible case is when $X$ has $\rho(X_{\overline{\mathbb{F}}_{q}})=20$. These are called singular K3 surfaces. Both singular and supersingular K3 surfaces are closely related to abelian surfaces, which enables us to verify Parshin's conjecture for them. More precisely:

\begin{prop}
Let $X$ be a smooth projective surface over $\mathbb{F}_{q}$, $p=\mathrm{char}(\mathbb{F}_{q})$. Then $X$ satisfies Parshin's conjecture in the following cases:
\begin{enumerate}[i)]
\item $X$ is an abelian surface,
\item $X$ is a Kummer surface,
\item $X$ is a supersingular K3 surface,
\item $X$ is a singular K3 surface and $p\geq 3$.
\end{enumerate}
\end{prop}
\begin{proof}
Let $\mathcal{M}_{\mathrm{Ab}}$ denote the sub-category of the category of Chow motives (with rational coefficients) generated by abelian varieties and Artin motives. By \cite[Corollaire 2.2]{Kah03}, smooth projective varieties $X$ with Chow motive $h(X)\in\mathcal{M}_{\mathrm{Ab}}$ (i.e. $X$ of abelian-type) satisfy Parshin's conjecture if they satisfy the Tate conjecture (see Remark \ref{Kahn on Parshin}). The Tate conjecture is known for the classes of variety in the proposition (\cite{Tat66} for $i)$, \cite{Cha13}, \cite{MP15} for K3 surface in general, although it has been known in the special cases of $ii)$, $iii)$ and $iv)$ for much longer). If $X$ is an abelian surface or a Kummer surface then clearly $h(X)\in\mathcal{M}_{\mathrm{Ab}}$, which proves $i)$ and $ii)$. For $iii)$, if $p=2$ then $X$ is unirational \cite{RS78}. If $p\geq 3$ it is not known that $X$ is unirational so we need a different argument. It is known that $X_{\overline{\mathbb{F}}_{q}}$ is derived equivalent to the supersingular K3 surface $Y$ with Artin invariant $1$, which is a Kummer surface. Indeed, by \cite[Proposition 5.2.5]{BL19}, one can obtain $X_{\overline{\mathbb{F}}_{q}}$ from $Y$ by iteratively taking moduli spaces of twisted sheaves. But derived equivalent twisted K3 surfaces have isomorphic Chow motives by \cite[Theorem 2.1]{Huy19}, so $h(X_{\overline{\mathbb{F}}_{q}})\simeq h(Y)\in\mathcal{M}_{\mathrm{Ab}}$, and therefore $X$ is of abelian-type by \cite[Theorem 1]{Via17}. For $iv)$ observe that $X_{\overline{\mathbb{F}}_{q}}$ sits in a Shioda-Inose structure; there exists an (ordinary) abelian surface $A$ over $\overline{\mathbb{F}}_{q}$ and dominant rational maps to and from the Kummer surface $\mathrm{Km}(A)$ which are both generically finite of degree $2$
\begin{equation*}
\mathrm{Km}(A)\stackrel{2:1}{\dashrightarrow} X_{\overline{\mathbb{F}}_{q}}\stackrel{2:1}{\dashrightarrow}\mathrm{Km}(A)\,.
\end{equation*}
Moreover, the abelian surface $A$ (and hence $\mathrm{Km(A)}$) can be defined over a finite field. For all of this and more see \cite[\S2.3]{Lie15}. In particular, there exists a finite extension $F/\mathbb{F}_{q}$ such that $X_{F}$ sits in a Shioda-Inose structure:
\begin{equation*}
\mathrm{Km}(A)\stackrel{2:1}{\dashrightarrow} X_{F}\stackrel{2:1}{\dashrightarrow}\mathrm{Km}(A)
\end{equation*}
with $A$ defined over $F$. One gets a surjective morphism of smooth projective varieties $A'\twoheadrightarrow X_{F}$ by blowing-up to resolve the dominant rational map $\mathrm{Km}(A)\dashrightarrow X_{F}$, so $h(X_{F})$ is a direct summand of $h(A')$ and hence $X_{F}$ is also of abelian-type. Finally, we deduce that $K_{n}(X)_{\mathbb{Q}}=0$ for all $n\geq 1$ by \'{e}tale descent for rational $K$-theory \cite[Theorem 2.15]{Tho85}. (Alternatively, one can conclude that $X$ is of abelian-type directly from \cite[Theorem 1]{Via17}).
\end{proof}

\begin{landscape}
\begin{center}
\begin{table}
\begin{adjustbox}{width=20cm}
\renewcommand{\arraystretch}{2.5}
\begin{tabular}{ |c||c|c|c|c|c|c|c|c| } 
 \hline
 $n$ & $H^{0}_{\mathcal{M}}(X,n)$ & $H^{1}_{\mathcal{M}}(X,n)$ & $H^{2}_{\mathcal{M}}(X,n)$ & $H^{3}_{\mathcal{M}}(X,n)$ & $H^{4}_{\mathcal{M}}(X,n)$ & $H^{5}_{\mathcal{M}}(X,n)$ & $H^{6}_{\mathcal{M}}(X,n)$ & $\cdots$ \\ 
 \hline
 \hline
$0$ & $\mathbb{Z}(X)$ & $0$ & $0$ & $0$ & $0$ & $0$ & $0$ & $\cdots$ \\
\hline
$1$ & $0$ & $\mathcal{O}^{\ast}(X)$ & $\mathrm{Pic}(X)$ & $0$ & $0$ & $0$ & $0$ & $\cdots$ \\ 
\hline
$2$ & $0$ & $\mathbb{Z}/(q^{2}-1)\mathbb{Z}$ & $\displaystyle\bigoplus_{\ell\neq p}{H^{1}_{\mathrm{\acute{e}t}}(X_{\overline{\mathbb{F}}_{q}},\mathbb{Z}_{\ell}(2))}_{\Gamma}$ & $\displaystyle\varinjlim_{r}H^{0}_{\mathrm{Zar}}(X,W_{r}\Omega_{X,\log}^{2})\oplus\bigoplus_{\ell\neq p}{H^{2}_{\mathrm{\acute{e}t}}(X_{\overline{\mathbb{F}}_{q}},\mathbb{Z}_{\ell}(2))}_{\Gamma}$ & $\mathrm{CH}_{0}(X)$ & $0$ & $0$ & $\cdots$ \\ 
\hline
$3$ & $0$ & $\mathbb{Z}/(q^{3}-1)\mathbb{Z}$ & $\displaystyle\bigoplus_{\ell\neq p}{H^{1}_{\mathrm{\acute{e}t}}(X_{\overline{\mathbb{F}}_{q}},\mathbb{Z}_{\ell}(3))}_{\Gamma}$ & $\displaystyle\bigoplus_{\ell\neq p}{H^{2}_{\mathrm{\acute{e}t}}(X_{\overline{\mathbb{F}}_{q}},\mathbb{Z}_{\ell}(3))}_{\Gamma}$ & $\displaystyle\bigoplus_{\ell\neq p}{H^{3}_{\mathrm{\acute{e}t}}(X_{\overline{\mathbb{F}}_{q}},\mathbb{Z}_{\ell}(3))}_{\Gamma}$ & $\mathbb{Z}/(q-1)\mathbb{Z}$ & $0$ & $\cdots$ \\ 
\hline
$4$ & $0$ & $\mathbb{Z}/(q^{4}-1)\mathbb{Z}$ & $\displaystyle\bigoplus_{\ell\neq p}{H^{1}_{\mathrm{\acute{e}t}}(X_{\overline{\mathbb{F}}_{q}},\mathbb{Z}_{\ell}(4))}_{\Gamma}$ & $\displaystyle\bigoplus_{\ell\neq p}{H^{2}_{\mathrm{\acute{e}t}}(X_{\overline{\mathbb{F}}_{q}},\mathbb{Z}_{\ell}(4))}_{\Gamma}$ & $\displaystyle\bigoplus_{\ell\neq p}{H^{3}_{\mathrm{\acute{e}t}}(X_{\overline{\mathbb{F}}_{q}},\mathbb{Z}_{\ell}(4))}_{\Gamma}$ & $\mathbb{Z}/(q^{2}-1)\mathbb{Z}$ & $0$ & $\cdots$ \\ 
\hline
$5$ & $0$ & $\mathbb{Z}/(q^{5}-1)\mathbb{Z}$ & $\displaystyle\bigoplus_{\ell\neq p}{H^{1}_{\mathrm{\acute{e}t}}(X_{\overline{\mathbb{F}}_{q}},\mathbb{Z}_{\ell}(5))}_{\Gamma}$ & $\displaystyle\bigoplus_{\ell\neq p}{H^{2}_{\mathrm{\acute{e}t}}(X_{\overline{\mathbb{F}}_{q}},\mathbb{Z}_{\ell}(5))}_{\Gamma}$ & $\displaystyle\bigoplus_{\ell\neq p}{H^{3}_{\mathrm{\acute{e}t}}(X_{\overline{\mathbb{F}}_{q}},\mathbb{Z}_{\ell}(5))}_{\Gamma}$ & $\mathbb{Z}/(q^{3}-1)\mathbb{Z}$ & $0$ & $\cdots$ \\ 
\hline
$6$ & $0$ & $\mathbb{Z}/(q^{6}-1)\mathbb{Z}$ & $\displaystyle\bigoplus_{\ell\neq p}{H^{1}_{\mathrm{\acute{e}t}}(X_{\overline{\mathbb{F}}_{q}},\mathbb{Z}_{\ell}(6))}_{\Gamma}$ & $\displaystyle\bigoplus_{\ell\neq p}{H^{2}_{\mathrm{\acute{e}t}}(X_{\overline{\mathbb{F}}_{q}},\mathbb{Z}_{\ell}(6))}_{\Gamma}$ & $\displaystyle\bigoplus_{\ell\neq p}{H^{3}_{\mathrm{\acute{e}t}}(X_{\overline{\mathbb{F}}_{q}},\mathbb{Z}_{\ell}(6))}_{\Gamma}$ & $\mathbb{Z}/(q^{4}-1)\mathbb{Z}$ & $0$ & $\cdots$ \\ 
\hline
\vdots & \vdots & \vdots & \vdots & \vdots & \vdots & \vdots & \vdots  & $\cdots$ \\ 
 \hline
 $n\geq 3$ & $0$ & $\mathbb{Z}/(q^{n}-1)\mathbb{Z}$ & $\displaystyle\bigoplus_{\ell\neq p}{H^{1}_{\mathrm{\acute{e}t}}(X_{\overline{\mathbb{F}}_{q}},\mathbb{Z}_{\ell}(n))}_{\Gamma}$ & $\displaystyle\bigoplus_{\ell\neq p}{H^{2}_{\mathrm{\acute{e}t}}(X_{\overline{\mathbb{F}}_{q}},\mathbb{Z}_{\ell}(n))}_{\Gamma}$ & $\displaystyle\bigoplus_{\ell\neq p}{H^{3}_{\mathrm{\acute{e}t}}(X_{\overline{\mathbb{F}}_{q}},\mathbb{Z}_{\ell}(n))}_{\Gamma}$ & $\mathbb{Z}/(q^{n-2}-1)\mathbb{Z}$ & $0$  & $\cdots$ \\ 
 \hline
 \vdots & \vdots & \vdots & \vdots & \vdots & \vdots & \vdots & \vdots & $\ddots$ \\ 
 \hline
\end{tabular}
\end{adjustbox}
\caption{The motivic cohomology groups of surfaces in Theorem \ref{main}.}
\label{fig:table}
\end{table}
\end{center}
\end{landscape}

\begin{landscape}
\begin{figure}
\centering
\begin{adjustbox}{width=20cm}
\begin{tikzpicture}
  \matrix (m) [matrix of math nodes,
    nodes in empty cells,nodes={minimum width=5ex,
    minimum height=5ex,outer sep=-5pt},
    column sep=1ex,row sep=10ex]{         
\ddots & \vdots & \vdots & \vdots & \vdots & \vdots & \vdots & \vdots & \vdots & \vdots & \vdots & \reflectbox{$\ddots$} \\         
         \cdots & 0 & 0 & 0 & 0 & 0 & \mathbb{Z} & 0 & 0 & 0 & 0 & \cdots \\
          \cdots & 0 & 0 & 0 & 0 & 0 & \mathcal{O}^{\ast}(X) & \mathrm{Pic}(X) & 0 & 0 & 0 & \cdots \\
     \cdots & 0 & 0 & 0 & 0 & \mathbb{Z}/(q^{2}-1)\mathbb{Z} & 0 & \displaystyle\varinjlim_{r}H^{0}_{\mathrm{Zar}}(X,W_{r}\Omega_{X,\log}^{2})\oplus\bigoplus_{\ell\neq p}{H^{2}_{\mathrm{\acute{e}t}}(X_{\overline{\mathbb{F}}_{q}},\mathbb{Z}_{\ell}(2))}_{\Gamma}& \mathrm{CH}_{0}(X) & 0 & 0 & \cdots \\
    \cdots & 0 & 0 & 0 & \mathbb{Z}/(q^{3}-1)\mathbb{Z} & 0 & \displaystyle\bigoplus_{\ell\neq p}{H^{2}_{\mathrm{\acute{e}t}}(X_{\overline{\mathbb{F}}_{q}},\mathbb{Z}_{\ell}(3))}_{\Gamma}& 0 & \mathbb{Z}/(q-1)\mathbb{Z} & 0 & 0 & \cdots \\
    \cdots & 0 & 0 & \mathbb{Z}/(q^{4}-1)\mathbb{Z} & 0 & \displaystyle\bigoplus_{\ell\neq p}{H^{2}_{\mathrm{\acute{e}t}}(X_{\overline{\mathbb{F}}_{q}},\mathbb{Z}_{\ell}(4))}_{\Gamma}& 0 & \mathbb{Z}/(q^{2}-1)\mathbb{Z} & 0 & 0 & 0 & \cdots \\ 
    \cdots & 0 & \mathbb{Z}/(q^{5}-1)\mathbb{Z} & 0 & \displaystyle\bigoplus_{\ell\neq p}{H^{2}_{\mathrm{\acute{e}t}}(X_{\overline{\mathbb{F}}_{q}},\mathbb{Z}_{\ell}(5))}_{\Gamma}& 0 & \mathbb{Z}/(q^{3}-1)\mathbb{Z} & 0 & 0 & 0 & 0 & \cdots \\
    \reflectbox{$\ddots$} & \vdots & \vdots & \vdots & \vdots & \vdots & \vdots & \vdots & \vdots & \vdots & \vdots & \ddots \\};
  \draw[-stealth] (m-3-7.south east) -- (m-4-9.north west)node[midway,above right] {$d_{2}$};
\end{tikzpicture} 
\end{adjustbox}
\caption{$E_{2}$-page of the Atiyah-Hirzebruch spectral sequence in Theorem \ref{AHSS}.} \label{fig:M1}
\end{figure} 
\end{landscape}

\begin{landscape}
\begin{table}
\begin{adjustbox}{width=20cm}
\renewcommand{\arraystretch}{2.5}
\begin{tabular}{ |c||c|c|c|c|c|c|c|c| } 
 \hline
 $n$ & $H^{0}_{\mathcal{M}}(X,n)$ & $H^{1}_{\mathcal{M}}(X,n)$ & $H^{2}_{\mathcal{M}}(X,n)$ & $H^{3}_{\mathcal{M}}(X,n)$ & $H^{4}_{\mathcal{M}}(X,n)$ & $H^{5}_{\mathcal{M}}(X,n)$ & $H^{6}_{\mathcal{M}}(X,n)$ & $\cdots$ \\ 
 \hline
 \hline
$0$ & $\mathbb{Z}(X)$ & $0$ & $0$ & $0$ & $0$ & $0$ & $0$ & $\cdots$ \\
\hline
$1$ & $0$ & $\mathcal{O}^{\ast}(X)$ & $\mathrm{Pic}(X)$ & $0$ & $0$ & $0$ & $0$ & $\cdots$ \\ 
\hline
$2$ & $0$ & $\mathbb{Z}/(q^{2}-1)\mathbb{Z}$ & $0$ & $\displaystyle\bigoplus_{\ell\neq p}{(\mathrm{Pic}(X_{\overline{\mathbb{F}}_{q}})\otimes\mathbb{Z}_{\ell}(1))}_{\Gamma}$ & $\mathrm{CH}_{0}(X)$ & $0$ & $0$ & $\cdots$ \\ 
\hline
$3$ & $0$ & $\mathbb{Z}/(q^{3}-1)\mathbb{Z}$ & $0$ & $\displaystyle\bigoplus_{\ell\neq p}{(\mathrm{Pic}(X_{\overline{\mathbb{F}}_{q}})\otimes\mathbb{Z}_{\ell}(2))}_{\Gamma}$ & $\mathbb{Z}/2\mathbb{Z}$ & $\mathbb{Z}/(q-1)\mathbb{Z}$ & $0$ & $\cdots$ \\ 
\hline
$4$ & $0$ & $\mathbb{Z}/(q^{4}-1)\mathbb{Z}$ & $0$ & $\displaystyle\bigoplus_{\ell\neq p}{(\mathrm{Pic}(X_{\overline{\mathbb{F}}_{q}})\otimes\mathbb{Z}_{\ell}(3))}_{\Gamma}$ & $\mathbb{Z}/2\mathbb{Z}$ & $\mathbb{Z}/(q^{2}-1)\mathbb{Z}$ & $0$ & $\cdots$ \\ 
\hline
$5$ & $0$ & $\mathbb{Z}/(q^{5}-1)\mathbb{Z}$ & $0$ & $\displaystyle\bigoplus_{\ell\neq p}{(\mathrm{Pic}(X_{\overline{\mathbb{F}}_{q}})\otimes\mathbb{Z}_{\ell}(4))}_{\Gamma}$ & $\mathbb{Z}/2\mathbb{Z}$ & $\mathbb{Z}/(q^{3}-1)\mathbb{Z}$ & $0$ & $\cdots$ \\ 
\hline
$6$ & $0$ & $\mathbb{Z}/(q^{6}-1)\mathbb{Z}$ & $0$ & $\displaystyle\bigoplus_{\ell\neq p}{(\mathrm{Pic}(X_{\overline{\mathbb{F}}_{q}})\otimes\mathbb{Z}_{\ell}(5))}_{\Gamma}$ & $\mathbb{Z}/2\mathbb{Z}$ & $\mathbb{Z}/(q^{4}-1)\mathbb{Z}$ & $0$ & $\cdots$ \\ 
\hline
\vdots & \vdots & \vdots & \vdots & \vdots & \vdots & \vdots & \vdots  & $\cdots$ \\ 
 \hline
 $n\geq 3$ & $0$ & $\mathbb{Z}/(q^{n}-1)\mathbb{Z}$ & $0$ & $\displaystyle\bigoplus_{\ell\neq p}{(\mathrm{Pic}(X_{\overline{\mathbb{F}}_{q}})\otimes\mathbb{Z}_{\ell}(n-1))}_{\Gamma}$ & $\mathbb{Z}/2\mathbb{Z}$ & $\mathbb{Z}/(q^{n-2}-1)\mathbb{Z}$ & $0$  & $\cdots$ \\ 
 \hline
 \vdots & \vdots & \vdots & \vdots & \vdots & \vdots & \vdots & \vdots & $\ddots$ \\ 
 \hline
\end{tabular}
\end{adjustbox}
\caption{The motivic cohomology groups of an Enriques surface $X$ over $\mathbb{F}_{q}$.}
\label{fig:table2}
\end{table}
\end{landscape}

\begin{landscape}
\begin{figure}
\centering
\begin{adjustbox}{width=20cm}
\begin{tikzpicture}
  \matrix (m) [matrix of math nodes,
    nodes in empty cells,nodes={minimum width=5ex,
    minimum height=5ex,outer sep=-5pt},
    column sep=1ex,row sep=10ex]{         
\ddots & \vdots & \vdots & \vdots & \vdots & \vdots & \vdots & \vdots & \vdots & \vdots & \vdots & \reflectbox{$\ddots$} \\         
         \cdots & 0 & 0 & 0 & 0 & 0 & \mathbb{Z} & 0 & 0 & 0 & 0 & \cdots \\
          \cdots & 0 & 0 & 0 & 0 & 0 & \mathcal{O}^{\ast}(X) & \mathrm{Pic}(X) & 0 & 0 & 0 & \cdots \\
     \cdots & 0 & 0 & 0 & 0 & \mathbb{Z}/(q^{2}-1)\mathbb{Z} & 0 & \bigoplus_{\ell\neq p}{(\mathrm{Pic}(X_{\overline{\mathbb{F}}_{q}}\otimes\mathbb{Z}_{\ell}(1))}_{\Gamma}& \mathbb{Z} & 0 & 0 & \cdots \\
    \cdots & 0 & 0 & 0 & \mathbb{Z}/(q^{3}-1)\mathbb{Z} & 0 & \displaystyle\bigoplus_{\ell\neq p}{(\mathrm{Pic}(X_{\overline{\mathbb{F}}_{q}}\otimes\mathbb{Z}_{\ell}(2))}_{\Gamma}& \mathbb{Z}/2\mathbb{Z} & \mathbb{Z}/(q-1)\mathbb{Z} & 0 & 0 & \cdots \\
    \cdots & 0 & 0 & \mathbb{Z}/(q^{4}-1)\mathbb{Z} & 0 & \displaystyle\bigoplus_{\ell\neq p}{(\mathrm{Pic}(X_{\overline{\mathbb{F}}_{q}}\otimes\mathbb{Z}_{\ell}(3))}_{\Gamma}& \mathbb{Z}/2\mathbb{Z} & \mathbb{Z}/(q^{2}-1)\mathbb{Z} & 0 & 0 & 0 & \cdots \\ 
    \cdots & 0 & \mathbb{Z}/(q^{5}-1)\mathbb{Z} & 0 & \displaystyle\bigoplus_{\ell\neq p}{(\mathrm{Pic}(X_{\overline{\mathbb{F}}_{q}}\otimes\mathbb{Z}_{\ell}(4))}_{\Gamma}& \mathbb{Z}/2\mathbb{Z} & \mathbb{Z}/(q^{3}-1)\mathbb{Z} & 0 & 0 & 0 & 0 & \cdots \\
    \reflectbox{$\ddots$} & \vdots & \vdots & \vdots & \vdots & \vdots & \vdots & \vdots & \vdots & \vdots & \vdots & \ddots \\};
  \draw[-stealth] (m-3-7.south east) -- (m-4-9.north west)node[midway,above right] {$d_{2}$};
  \draw[-stealth] (m-4-6.south east) -- (m-5-8.north west);
  \draw[-stealth] (m-5-5.south east) -- (m-6-7.north west);
  \draw[-stealth] (m-6-4.south east) -- (m-7-6.north west);
\end{tikzpicture} 
\end{adjustbox}
\caption{$E_{2}$-page of the Atiyah-Hirzebruch spectral sequence of Enriques surface $X$ over $\mathbb{F}_{q}$.} \label{fig:M2}
\end{figure} 
\end{landscape}

\end{document}